\documentclass[12pt]{article}
\usepackage{graphicx}
\usepackage{amsmath, amsthm}
\usepackage{amsmath}
\usepackage{amsfonts}
\usepackage{amssymb}
\usepackage{mathrsfs}
\usepackage{color}

\newtheoremstyle{thm}{1.5ex}{1.5ex}{\itshape\rmfamily}{}
{\bfseries\rmfamily}{}{2ex}{}

\newtheoremstyle{rem}{1.3ex}{1.3ex}{\rmfamily}{}
{\itshape}
{} {1.5ex}{}

\theoremstyle{thm}

\newtheorem{thm}{Theorem}[section]

\newtheorem{cor}[thm]{Corollary}
\newtheorem{lemma}[thm]{Lemma}
\newtheorem{prop}[thm]{Proposition}

\theoremstyle{definition}

\newtheorem{hyp}[thm]{Hypothesis}
\newtheorem{remark}[thm] {Remark}

\newtheorem{defn}[thm]{Definition}

\definecolor{purple}{rgb}{0.65, 0, 1}
\definecolor{orange}{rgb}{1,.5,0}
\newcommand{\e}{\varepsilon}

\usepackage{enumerate}
\newcommand{\pc}[2]{\bar \Omega_{#1}({#2})}

\begin{document}

\title{Hamiltonian ODE's on a Space of Deficient Measures}
\vspace{-10cm}
\author{L.~Chayes$^*$, W.~Gangbo$^{**}$ and H.~K.~Lei$^{ \dagger *}$}
\vspace{-2cm}
\date{}
\maketitle
\vspace{-4mm}
\centerline{${}^*$\textit{Department of Mathematics, UCLA}}
\centerline{${}^{**}$\textit{Department of Mathematics, Georgia Institute of Technology}}
\centerline{${}^\dagger$\textit{Department of Mathematics, California Institute of Technology}}
\abstract{We continue the study (initiated in \cite{wilfrid_ambrosio}) of Borel measures whose time evolution is provided by an interacting Hamiltonian structure.  Here, the principal focus is the development and advancement of \textit{deficency} in the measure caused by displacement of mass to infinity in finite time.
We introduce -- and study in its own right -- a regularization scheme based on a dissipative mechanism which naturally degrades mass according to distance traveled (in phase space).  Our principal results are obtained based on some dynamical considerations in the form of a condition which forbids mass to return from infinity.

\setcounter{tocdepth}{2}
\tableofcontents

%
%
%
%
%
%
%
%

\section*{Introduction}

In this work, we further the study of time evolution for certain Borel measures whose underlying dynamic is provided by a \textit{Hamiltonian} structure.  
In particular, for a given measure $\mu_{t}$ on $D:= 2d$--dimensional \textit{phase space} for a fluid or particle system in $d$--dimensions, one is led to consider a ``Hamiltonian functional'',  $\mathscr H(\mu_t)$, whose gradient with respect to the Wasserstein metric (see Equation \eqref{grad_ham}) provides the velocity field which in turn evolves the measure.  While the mass conserved case has already been considered in e.g., \cite{wilfrid_ambrosio}, our concern in this note is the problem of Hamiltonian dynamics where the kinematics allow for the possibility that particles may reach infinity in finite time.  (As can be readily checked, dynamics subject to an external potential with a super--quadratic drop to negative infinity shall satisfy such a property.)  

Motivated by the fact that mass conserved dynamics can be described by an appropriate continuity equation,  as described in \cite{wilfrid_ambrosio}, we begin by developing a weak theory of a \emph{deficient} continuity equation where mass is degraded according to distance traveled in phase space.  It is worth emphasizing that this is a natural regularization for which degradation of mass is everpresent; moreover, this regularization allows us to look at infinite volume measures from the outset.  On the other hand, it is noted that while the work in \cite{wilfrid_ambrosio} employs the Wasserstein distance -- which is equivalent to vague convergence along with convergence of second moments, our result is much more modest: In the present note we shall content ourselves, ultimately, with distributional convergence, i.e., weak convergence restricted to finite volumes. 

This restriction to finite volumes induces certain dynamical considerations.  Indeed, it is almost a tautology that the dynamics cannot be well--described distributionally without some notion that particles cannot ``return from infinity''.  (Such a condition can be established in a variety of circumstances, the most trivial example being in the case of a radially symmetric potential, where the existence of infinitely many ``outward'' maxima clearly define regions of no return.)  A general version of such a condition, which we refer to as a dynamical hypothesis, will be described in Section \ref{d}.
With such a dynamical hypothesis in hand, we shall indeed be able to extract some limiting dynamics as our regularization parameter tends to zero and establish a mass convergence result as detailed in Section \ref{sec_mass_conv}.

%
%
%
%
%
%
%
%

\section{Preliminaries}

In this section we will introduce the deficient continuity equation and establish some basic properties.

\subsection{Deficient Equation and \emph{a priori} Estimates}\label{sec_def_cont}

\begin{defn}
We will denote by $\mathscr M$ the space of all finite Borel measures generated by the open sets in $\mathbb R^D$.  Given some $\alpha \geq 0$ and a Borel measure $\mu \in \mathscr M$, we define 
\[ M_{\alpha}(\mu) = \mathbb E_{\mu}(e^{\alpha |X|})\]
to be the $\alpha$--exponential moment of $\mu$.  We also let $\mathcal M_{\alpha}$ denote the set of all Borel measures with finite $\alpha$--exponential moment.  
\end{defn}


Before we state a preliminary existence result for fixed velocity fields, we will make the following observation:

\begin{prop}\label{prop:erz}
Let  $T > 0$ and let $v: [0,T] \times \mathbb R^D \rightarrow \mathbb R^D$ be a Borel vector field such that $v_t:=v(t, \cdot)$ is locally Lipschitz for every $t \in [0,T]$.  Let us assume that for every compact $K \subset \mathbb R^D,$ $f_K \in L^1(0,T),$ where 
\[ 
f_K(t)= \sup_{x \in K} |v_t| + \mathrm{Lip}(v_t, K).  
\] 
\begin{itemize}
\item[(i)] Then for all $x$ in $\mathbb R^D$, there exists a $\tau(x)>0$ such that there exists a (unique) solution to the ODE   
\[ \dot X_t = v_t(X_t), ~~~X_0 = x,\] 
on $[0, \tau(x))$. Further, either $\tau(x) = T$ or $t \mapsto |X_t|$ is unbounded on $[0, \tau(x))$, in which case 
\[ \int_0^{\tau(x)} |\dot X_t|~dt = \infty.\]   
\item[(ii)] The function $\tau: \mathbb R^D \rightarrow (0,\infty)$ is lower semicontinuous and for $t \in (0,T)$ the function $X_t$ is one-to-one on  the open set 
\[\mathbb S_t:=\{ x \in \mathbb R^D \mid \tau(x)>t\}.\] 
\end{itemize}
\end{prop}
\begin{proof}
For each $x$, existence up to some maximal time $\tau(x)$ follows by standard ODE theory given the assumption on $v_t$.  Furthermore, the solution trajectories are continuous.  If $\tau(x)<T$ and $t \mapsto |X_t|$ is bounded on $[0, \tau(x))$  then $t \mapsto X_t$ is Lipschitz on that set and so, it admits a left limit $X_{\tau(x)}$ at $\tau(x).$ Then the identity $X_t=X_0+ \int_0^t v_s \circ X_s~ds$ can first be extended to $t=\tau(x)$ and then  beyond. In particular, if $|X_{\tau(x)}|$ is bounded, then we could extend the solution of the ODE to a larger interval, contradicting the maximality of $\tau(x)$. 

Let $D_b$ denote the ball of radius $b$ around the origin, and for $t<T$ set 
\[ \mathbb B_t(b) = \{x \in \mathbb R^D \mid \tau(x)>t, X(x, \cdot)[0,t] \subset D_b\}.\] 
We shall next demonstrate that 
$\mathbb B_{t}(b)$ is an open set.  Let $x \in \mathbb B_t(b)$ and $y$ satisfying $|x - y| < \delta$ with $\delta$ to be specified momentarily: For $L > 0$ suppose that $\delta$ has been chosen so that 
\begin{equation}\label{eq:conditionBb1} \delta < L e^{-\int_0^t f_{D_{b + L}}(s)~ds}. \end{equation}
We claim that under these conditions, $y \in \mathbb B_{t}(b+L)$.  Indeed, clearly $y \in \mathbb B_{\vartheta}(b+L)$ for $\vartheta > 0$ sufficiently small; let us suppose that $\vartheta$ is maximal and assume towards a contradiction that $\vartheta < t$.  Letting $r(t): = |X_t(x) - X_t(y)|$ we directly see that for any $t' < \vartheta$, 
\[ \frac{dr}{dt} (t') \leq  r(t') f_{D_{b+L}}(t').\]
Thus, certainly, by Gronwall's inequality $r(t') < L$ for any $t' \leq \vartheta$, but continuity then also implies that $r(\vartheta + \eta) < L$ for some sufficiently small $\eta > 0$, contradicting the maximality of $\vartheta$.  We conclude that $\vartheta \geq t$. So $y \in \mathbb B_{t}(b+L)$.  Now, since $X(x, \cdot)[0, t]$ is compact we have that $x \in \mathbb B_{t}(b')$  for some $b' < b$; finally, choosing $L = b- b'$ we have $y \in \mathbb B_t(b)$ and we have established that $\mathbb B_t(b)$ is an open set.  

As for the semicontinuity, let  $x \in \mathbb R^D$ and choose an arbitrary positive number $a$ smaller than $\tau(x)$ and then choose $t \in (\tau(x)-a, \tau(x))$.   As $X(x, \cdot)[0,t]$ is a compact set, it is contained in a ball $D_{b}$. Choose $L=1$, say, and $\delta$ -- with $\delta$ as in  Equation (\ref{eq:conditionBb1}) (with $\delta \ll 1$). Since 
$D_{\delta}(x)$, the open ball of center $x$ and radius $\delta$, is contained in $\mathbb B_t(b+1)$ and we have $\tau \geq t> \tau(x)-a$ on $D_{\delta}(x)$. This proves that $\liminf_{y \rightarrow x} \tau(y) \geq \tau(x)-a. $ As $a$ is arbitrary we conclude that $\liminf_{y \rightarrow x} \tau(y) \geq \tau(x).$ Thus $\tau$ is lower semicontinuous
 and, moreover, it follows that $\mathbb S_t$ an open set. 

Finally, if $x,y \in \mathbb S_t$ are such that $X_t(x)=X_t(y)$ then the functions $s \rightarrow r_x(s):=X_{t-s}(x)$ and $s \rightarrow r_y(s):=X_{t-s}(y)$ satisfy 
\[r_x(0)=r_y(0), \quad \dot r_x= w_s(r_x), \quad \dot r_y= w_s(r_y), \qquad w_s:=-v_{t-s}\]
on $[0,t]$. Hence, $r_x(t)=r_y(t)$ i.e. $x=y,$ which proves that $X_t$ is one-to-one on $\mathbb S_t.$ 
\end{proof}

\begin{remark}\label{measurability} Let $v$ denote a velocity field satisfying the hypothesis of Proposition \ref{prop:erz}. 

(i) To ensure that $f_K$ is Borel measurable it suffices to assume that $v(\cdot, x)$ is Borel measurable for every $x \in  \mathbb R^D.$ 

(ii) Let $\mathcal O$ be the set of $(x,t)$ such that $0<t<\tau(x)$. The lower semicontinuity of $\tau$ ensures that $\mathcal O$ is an open subset of $\mathbb R^D \times (0,\infty).$ Also, $(x,t) \rightarrow X_t(x)$ is continuous on $\mathcal O$.
\end{remark}




%
%
\begin{prop}\label{preserv_exp}
For fixed $\alpha_0 > 0$, $\alpha > 0$, and $0 < T < \infty$, let $\mu_0 \in \mathcal M_{\alpha_0}$ denote some initial Borel measure on $\mathbb R^D$, and let $v_t$ denote a velocity field satisfying the hypothesis of Proposition \ref{prop:erz}.  Then 

\begin{itemize}
\item[(i)] there exists $(\mu_t)_{t \in [0, T]}$ such that 
\begin{equation} \label{partial_cont} \partial_t \mu_t + \nabla \cdot (v_t \mu_t) = -\e |v_t| \mu_t,\end{equation}
in the sense of distribution:
\[ \forall \varphi \in C_c^\infty(\mathbb R^D \times (0, T)),~~~\int_0^T \int_{\mathbb R^D} (\partial_t \varphi + \langle v_t, \nabla \varphi \rangle)~d\mu_t~dt = -\int_0^T \int_{\mathbb R^D} \e |v_t| ~\varphi ~d\mu_t~dt;\]  

\item[(ii)] the measure $\mu_t$ is supported by $X_t(\mathbb S_t)$ and 
$$ 
\int_{\mathbb R^D} \psi ~d\mu_t= \int_{\mathbb S_t} (R_t \psi) \circ X_t   ~d\mu_0,
$$
where $R_t$ is defined as in Equation \eqref{defn_Rt};

\item[(iii)] furthermore, if $\alpha \leq \min\{\alpha_0, \e\}$, $M_{\alpha}(\mu_t)$ is monotonically nonincreasing in $t$. In particular the total mass $M_{0}(\mu_t)$  is monotonically nonincreasing in $t$. 
\end{itemize}
\end{prop}
\begin{proof}
For $b$ be a positive integer we select a map $\phi^b: \mathbb R^D \rightarrow D_{b+1}$ of class $C^2$  of Lipschitz constant less than or equal to $b$ such that $\phi^b(x)=x$ for $|x| \leq b,$  and $\phi^b(x)=0$ for $|x| \geq 2+b.$ Let $X_t^b$ be the solution   of the ODE as described in Proposition \ref{prop:erz} when $v$ is substituted by $v^b= v \circ \phi^b.$ Note that  as 
\[|\dot X_t^b| \leq f_{D_{b+1}}, \quad  f_{D_{b+1}} \in L^1(0,T),\] 
$X^b_t$ exists for all $t \in [0,T]$ and  is invertible.  
Let $\mathbb B_t(b)$ be as above and observe that 
\begin{equation}\label{eq:flowsconcide} 
x \in  \mathbb B_t(b), \quad s \in [0,t] \quad \implies \quad X^b_s(x)= X_s(x).
\end{equation} 
Indeed, for $x \in  \mathbb B_t(b)$ and $s \leq t$ 
\[X_s(x)=x + \int_0^s v_\tau (X_\tau(x))d\tau=x + \int_0^s v_\tau^b (X_\tau(x))d\tau . \]
Set 
$$
\mu_t^{*, b}= X^b_{t} \# \mu_0, \quad \mu_t^b = R_t^b\mu_t^{*, b} 
$$ 
where $R^b_t$ is defined on $[0,T] \times \mathbb R^D$ by 
$$
R^b_t \circ X^b_t = \exp\Bigl(-\int_0^t \e | \dot X^b_s|~d\tau \Bigr).
$$
Similarly, for $t<T$ and $x \in \mathbb S_t$ we define 
\begin{equation}\label{defn_Rt}
R_t \circ X_t = \exp\Bigl(-\int_0^t \e | \dot X_s|~d\tau \Bigr).
\end{equation}
As 
\[ \partial_t \mu^b_t + \nabla \cdot (v^b_t R^b_t) = -\e |v^b_t| R^b_t \quad \hbox{and} \quad 
\partial_t \mu^{*, b}_t + \nabla \cdot (v^b_t \mu^{*, b}_t) =0\]  
in the sense of distribution on $(0,T) \times \mathbb R^D$, 
\begin{equation} \label{partial_contb}  \partial_t \mu^b_t + \nabla \cdot (v^b_t \mu^b_t) = -\e |v^b_t| \mu^b_t\end{equation}
in the sense of distribution on $(0,T) \times \mathbb R^D.$ We use Equation (\ref{eq:flowsconcide}) to observe for $t<T$ 
\begin{equation}\label{eq:flimitofRb} 
\lim_{b \rightarrow \infty} R^b_t \circ X^b_t (x)= 
 \begin{cases} R_t \circ X_t (x) & \text{if $t< \tau(x)$,}
\\
0 & \text{if $t \geq \tau(x)$}
\end{cases}.  
\end{equation} 
As $(t,x) \rightarrow  R^b_t \circ X^b_t (x)$ is continuous the limiting function in Equation (\ref{eq:flimitofRb}) is a Borel function. For $t \in [0,T)$ we define the measure $\mu_t$ supported by $X_t(\mathbb S_t)$ by 
$$
\int_{\mathbb R^D} \psi d\mu_t= \int_{\mathbb S_t} (R_t \psi) \circ X_t   ~d\mu_0
$$ 
for $\psi \in C_b(\mathbb R^D).$  By Equations (\ref{eq:flowsconcide}) and (\ref{eq:flimitofRb}) 
\begin{equation}\label{eq:flimitofmub} 
\lim_{b \rightarrow \infty} \int_0^T \Bigl( \int_{\mathbb R^D} \varphi(t,y) ~d\mu_t^b(y) \Bigr) ~dt=   
\int_0^T \Bigl( \int_{\mathbb R^D} \varphi(t,y) ~d\mu_t (y) \Bigr) ~dt 
\end{equation} 
for $\varphi \in C_b([0,T] \times \mathbb R^D).$

{\bf 1.} Claim: $(\mu)_{t \in [0, T]}$ satisfies Equation (\ref{partial_cont}).

{\it \hskip 0.2in Proof of claim $1$.} In light of Equation (\ref{partial_contb}) it suffices to show for arbitrary $\varphi \in C_c^1((0,T) \times \mathbb R^D)$ that  
\[ \lim_{b \rightarrow \infty} \int_0^T \Bigl( \int_{\mathbb R^D} \langle \nabla \varphi, v^b \rangle ~d\mu_t^b(y) \Bigr) dt=   
\int_0^T \Bigl( \int_{\mathbb R^D} \langle \nabla \varphi, v \rangle ~d\mu_t (y) \Bigr) ~dt \] 
and 
\[ \lim_{b \rightarrow \infty} \int_0^T \Bigl( \int_{\mathbb R^D}  \varphi |v^b| ~d\mu_t^b(y) \Bigr) ~dt=   
\int_0^T \Bigl( \int_{\mathbb R^D}  \varphi | v | ~d\mu_t (y) \Bigr) ~dt. \]  
Let $r > 0$ be chosen so that -- say -- the set $[r,T-r]\times B_{1/r}$ contains the support of $\varphi$.  Now let $\omega$ be such that $|v| < \omega$ on $[r,T-r]\times B_{1/r}$.  Then once $b > \omega$, we have $\langle \nabla \varphi, v^b \rangle=\langle \nabla \varphi, v \rangle$ and so, Equation (\ref{eq:flimitofmub}) yields the first identity of the claim. We obtain the second identity in a similar manner.

In the above, the left hand sides actually \emph{equal} the right hand sides once $b$ is large enough so that $D_b$ subsumes the support of $\varphi$ since then $v^b = v$ and we may apply Equation \eqref{eq:flimitofmub}.

{\bf 2.} Claim: Let $\alpha \leq \min\{\alpha_0, \e\}$. Then $M_{\infty, \alpha}(\mu_t)$ is monotonically nonincreasing in $t$

{\it \hskip 0.2in Proof of claim $2$.} Let $0 \leq t_1<t_2 \leq T$. As  $v^b$ is of compact support Equation (\ref{partial_contb}) implies 
\begin{equation} \label{partial_cont.ter} 
{d \over dt} \int \psi ~d\mu_t^b= \int (\langle \nabla \psi, v^b_t \rangle -\e \psi |v^b_t|) ~d \mu_t.\end{equation}
so with $\psi = e^{\alpha |x|}$, we get  
\[ {d \over dt} M_{\infty, \alpha}(\mu_t^b)=\int |v^b_t| e^{\alpha |y|} \Bigl( \alpha \langle  {y\over |y|}, {v^b_t\over |v^b_t|} \rangle -\e \Bigr) ~d\mu_t \leq 0\] 
holds if $\alpha \leq \e$, and hence, 
\begin{equation} \label{partial_cont.bis} 
M_{\infty, \alpha}(\mu_{t_2}^b) \leq M_{\infty, \alpha}(\mu_{t_1}^b).
\end{equation}
As  $\alpha \leq \e$ we have 
\[  R^b_t \circ X^b_t (x)  \leq   
\exp -\alpha \Bigl(  \int_0^t  |\dot X^b_s(x)| ~ds \Bigr) 
 \]
and so, the inequality  
\[ |X^b_t(x)-x| \leq \int_0^t |\dot X^b_s(x)| ~ds\] yields  
\[  R^b_t \circ X^b_t (x) e^{\alpha |X^b_t (x)|} \leq   e^{\alpha |x|} 
 \]
Since $x \rightarrow \exp(\alpha |x|)$ belongs to $L^1(\mu_0)$, we may use  Equations (\ref{eq:flowsconcide}) and (\ref{eq:flimitofRb}) and apply the Lebesgue dominated convergence theorem to obtain that 
\[  
\lim_{b \rightarrow \infty} \int_{\mathbb R^D} R^b_t \circ X^b_t (x) e^{\alpha |X^b_t (x)|} ~d\mu_0(x)= 
\int_{\mathbb S_t} R_t \circ X_t (x) e^{\alpha |X_t (x)|} ~d\mu_0(x)
\] 
which shows that $M_{\infty, \alpha}(\mu_{t}^b)$ tends to $M_{\infty, \alpha}(\mu_{t})$ as $b$ tends to $\infty.$ This together with (\ref{partial_cont.bis}) proves the claim. 
\end{proof}
%
%
\begin{remark}\label{re:bounded_moment} 
We make some remarks on some (almost) automatic extensions of these results.

(i) The fact that the $\alpha$-exponential moment of the solution of (\ref{partial_cont}) decreases in time ensures that we can repeat Proposition \ref{preserv_exp} on the interval $[T,2T]$, $\cdots$, $[nT,(n+1)T]$ to obtain that Equation (\ref{partial_cont}) is satisfied on $[0,\infty) \times \mathbb R^D.$

(ii) If $v$ is a velocity field satisfying the hypothesis of Proposition \ref{prop:erz} and only the $m^{\text{th}}$ moment of $\mu_0$ is finite (i.e., $\mu_0$ may not have an $\alpha$--exponential moment) then the $m^{\text{th}}$ moment of $\mu_t$ stays bounded on $[0,T).$ (In the particular case where $m=0$ the total mass of $\mu_t$ is less than or equal to that of $\mu_0$.)  Indeed,  let us define 
\[ S_t(x) = \int_0^t |\dot X_\tau(x)|~d\tau. \]
Then 
\[ \int_{\mathbb R^D} |y|^m~d\mu_t(y) =\int_{\mathbb S_t} e^{-\e S_t(x)} |X_t (x)|^m ~d\mu_0(x)
\leq \int e^{-\e S_t(x)} (|x| + S_t(x))^m ~d\mu_0 \]
Now we divide the integral into $\{x: S_t(x) < |x|\}$ and $\{x: S_t(x) > |x|\}$.  The integral over the former region is bounded by $2^m$ times the initial moment, whereas the integral over the latter region is bounded by a constant depending on $\e.$  

(iii) Assume $\alpha$, $v$ and $\mu_0$ are as in Proposition \ref{preserv_exp}. Let $t \rightarrow \mu_t$ be the solution obtained in that proposition and let $\varphi \in C^1(\mathbb R^d)$.  Then we claim, that after some computations along the lines of the proof of said Proposition, that 
\begin{equation}\label{e:cont_path1}
\left|\int_{\mathbb R^D} \varphi ~d\mu_{t_2}-\int_{\mathbb R^D} \varphi ~d\mu_{t_1}\right| \leq (1+\e) \|\varphi\|_{C^1}\int_{t_1}^{t_2} \int_{\mathbb R^D} |v_s| ~d\mu_s
ds.
\end{equation} 
Indeed, we can first obtain the inequality in Equation (\ref{e:cont_path1}) for $\varphi \in C^1_c(\mathbb R^D).$  An approximation argument then yields the general case.
\end{remark} 
\subsection{Limiting Measures and Equations}\label{sec_lim_general}

Let $\mu_t^\varepsilon$ denote a (distributional) solution to the deficient continuity equation 
\begin{equation}\label{deficient_cont_eq} \partial_t\mu_t^\e + \nabla \cdot (\mu_t^\e v_t^\e) = -\varepsilon|v_t^\e|\mu_t^\e.\end{equation}
We will now establish existence of the necessary $\varepsilon \rightarrow 0$ limiting measures. We remark that here we will retrieve the limit abstractly, making no statement about the limiting dynamics.  We will address \emph{Hamiltonian} dynamics in the following section.

As before, we denote by $X_t$ the characteristic in Equation \eqref{deficient_cont_eq}: $\dot X_t = v_t(X_t)$, $X_0 = \text{Id}$. 

%
\begin{remark}\label{dist_narrow}
We point out that in this section and the next, we will use the weakest form of convergence of measures: Distributional convergence, i.e., $\mu_n \rightharpoonup \mu$ if 
\[\forall \varphi \in C_c^\infty(\mathbb R^D), ~~~ \lim_{n\rightarrow \infty} \int \varphi~d\mu_n = \int \varphi~d\mu.\] 
However,  if $\{\int \psi~d\mu_n\}_{n \in \mathbb N}$ is bounded for some nonnegative $\psi \in C(\mathbb R^D)$ such that $\psi(x)$ tends uniformly to $\infty$ as $|x|$ tends to $\infty$ (e.g., a moment condition) then distributional convergence is equivalent to narrow convergence which is defined as $\mu_n \underset{^n}{\rightharpoonup}\mu$ if
\[ \forall \varphi \in C_b(\mathbb R^D),~~~ \lim_{n\rightarrow \infty} \int \varphi~d\mu_n = \int \varphi~d\mu.\]  
So we will often (automatically) acquire narrow convergence but utilize $C_c^\infty$ functions to carry out the relevant arguments. 

We shall also use weak$^\ast$ convergence which is defined as $\mu_n \underset{^\ast}{\rightharpoonup} \mu$ if
\[ \forall \varphi \in C_0(\mathbb R^D),~~~ \lim_{n\rightarrow \infty} \int \varphi~d\mu_n = \int \varphi~d\mu.\]    
Since our measures are no longer probability measures, we prefer to not speak of tightness but instead think of them as Radon measures and abstractly extract a narrow limit point -- which may very well have mass much less than the sequence from which it originated but is none the less a Radon measure.  (Indeed, by the Riesz representation theorem, the dual of $C_0(\mathbb R^D)$ is isometrically isomorphic to the space of all Radon measures.)  Then if we wish to establish some property of the limiting measure (e.g., that it satisfies some suitable equation) it is enough to work with functions in $C_c^\infty(\mathbb R^D)$, as was discussed in the previous paragraph.  

Finally, as far as convergence of measures of sets are concerned, by standard properties of Radon measures it is the case that if $\mu_n \underset{^\ast}{\rightharpoonup} \mu$ and $A$ is a Borel set, then 
\begin{equation}\label{set_conv} \mu(A^\circ) \leq \liminf_n \mu_n(A) \leq \limsup_n \mu_n(A) \leq \mu(\bar A).\end{equation}
We refer the reader to e.g., \cite{EG}, Chapter 1.9 for such results.   
\end{remark}

First let us extract a narrow continuity statement for measures satisfying the deficient continuity equation:
%
%
\begin{lemma}\label{lemma_cont_def_cont} Let $v^\e_t$ be as in Proposition \ref{prop:erz} and let $\e \in (0,1)$. Suppose $(\mu_t^\e)_{t \in [0, T]}$ satisfies the deficient continuity Equation \eqref{deficient_cont_eq}.  Then $t \rightarrow \mu_t^\e$ is a continuous path in $\mathscr M$, when the latter space is endowed with the distributional convergence topology. Moreover, if the $\alpha^{\mbox{th}}$ moment of $\mu_0^\e$ is finite for some $\alpha>0$ then $t \rightarrow \mu_t^\e$ is a continuous path in $\mathscr M$ for the narrow convergence topology.
\end{lemma}
\begin{proof} Let $\varphi \in C_c^\infty(\mathbb R^D)$. By Equation (\ref{partial_cont}) the distributional derivative of $t \rightarrow g^\e_\varphi(t):=\int \varphi~d\mu_t^\e$ exists and is equal to
\[\int_{D \cap X_t^\e(\mathbb S_t)} (\langle \nabla \varphi, v_t^\e \rangle -\e |v_t^\e| \varphi)~d\mu_t^\e  \]
where $D$ is an open ball containing the support of $\varphi$ and $\mathbb S_t$ is defined in Proposition \ref{prop:erz}, corresponding to $v_t^\e$. Let 
\begin{equation}\label{eq:vel-upperbound} k_{D}^\e(t)=\sup_{x \in D} |v^\e_t|.\end{equation} 
We have 
\[ |\langle \nabla \varphi, v_t^\e \rangle -\e |v_t^\e| \varphi| \leq C k_{D}^\e(t) \]
for a constant $C$ depending only on $\varphi$ but independent of $t$.  As $k_{D}^\e \in L^1(0,T)$ we conclude that 
$g^\e_\varphi \in W^{1,1}(0,T)$ and so, it is continuous.  

Having established continuity in the distributional topology, continuity in the narrow topology is readily established under the stated conditions by standard approximation arguments.  Indeed, for $\varphi \in C_b^\infty$ we may write 
$\varphi = \varphi_n + \varphi - \varphi_n$ with 
$\varphi_n \in C_c^\infty$ satisfying $\varphi_n \leq \varphi$ and 
$\varphi - \varphi_n$ supported only outside a large region (tending to all of $\mathbb R^D$
as $n \to \infty$).  Remark \ref{re:bounded_moment} ensures if 
the $m^{\text{th}}$ moment of $\mu_0^\e$ is finite for some $m>0$
then  the $m^{\text{th}}$ moment of $\mu_t^\e$ remains uniformly bounded on $[0,T).$
This provides us with a uniform tightness condition that can be used to estimate the 
``non--compact'' portion of $\varphi$.
\end{proof}
%
%
\begin{remark}\label{rem_finite} Further suppose in Lemma \ref{lemma_cont_def_cont} that, e.g., $|\mu^\e_t| \leq 1$ and that for any compact set $K \subset \mathbb R^D$, there exists some constant $C(K) > 0$ such that for all $\e$,
\begin{equation}\label{poly_bound} \sup_{t \in [0, T], ~x \in K}|v_t^\e(x)| \leq C(K).
\end{equation}
That is, the velocity fields are locally bounded, uniformly for all $\e$ and $t$ of interest.  

Let us record that in the proof of that Lemma we have exhibited a functional $\varphi \rightarrow C_\varphi$ of $C_c^\infty(\mathbb R^D)$ into $(0,\infty)$ such that if $D$ is an open ball containing the support of $\varphi$ then 
\[||g^\e_\varphi||_{W^{1,1}(0,T)} \leq C_\varphi \int_0^T k_D(t)~dt, \qquad 
g^\e_\varphi(t):=\int \varphi~d\mu_t^\e. \]
But since $k_D \in L^\infty(0,T)$, we have in fact that $g_\varphi^\e \in W^{1, \infty}(0,T)$ for all $\e$ and hence it is Lipschitz, with Lipschitz constant $\|k_D\|_{L^\infty}$.  Thus, it is emphasized, that if the velocity field is bounded uniformly in $\e$, so is the resulting estimate on the relevant time derivative.
\end{remark}

\begin{prop}\label{narrow_conv}
Suppose $\sup_\e|\mu_0^\e|<\infty$ and we have $v_t^\e$ and $\mu_t^\e$ such that Equations \eqref{deficient_cont_eq} and \eqref{poly_bound} hold.  Then there exists a sequence $(\e_n)$ decreasing to $0$ such that $(\mu^{\e_n}_t)$ has a distributional limits $\mu_t$ as $n$ tends to $\infty$, for all $t \in [0,T].$ Furthermore, $t \rightarrow \mu_t$ is  continuous for the distributional convergence. 
\end{prop}
\begin{proof} Assume for instance that $|\mu_0^\e| \leq 1.$ Remark \ref{re:bounded_moment} ensures that $|\mu_t^\e| \leq 1$ uniformly in $t$ and $\e.$ Using a diagonal sequence argument we can apply the Banach--Alaoglu theorem to obtain  $(\e_n) \subset (0,\infty)$ converging to $0$ as $n$ tends to $\infty$ such that $(\mu_t^{\e_n})$ converges in the distributional sense to some $\mu_t \in \mathscr M$ for e.g., each $t \in \mathscr D:=(0,T) \cap \mathbb Q.$

{\bf 1.} Claim: If $t \in (0,T)$ then for any sequence $(t_k) \subset \mathscr D$ converging to $t$ we have that $\mu_{t_k} \rightharpoonup \mu_t$ for some $\mu_t \in \mathscr M$ independently of the sequence $(t_k).$  (That is, the limit can be extended to all $t \in[0, T]$.)

{\it \hskip 0.2in Proof of claim.} By the Banach--Alaoglu theorem the set $(\mu_t)_{t \in \mathscr D }$ is pre-compact for the distributional topology. Let $(t_k), (t^\ast_k)\subset \mathscr D$ be sequences converging to $t$ as $k$ tends to $\infty$ and suppose that $\mu_{t_k} \rightharpoonup \nu$ and $\mu_{t_k^\ast} \rightharpoonup \nu^\ast$ as $k$ tends to $\infty$. Let $D$ be an open ball of radius $r$ containing the support of an arbitrarily fixed function $\varphi \in C_c^\infty(\mathbb R^D)$, set $f_D(t)=C(D)$ and as in Remark \ref{rem_finite} set  \[g^\e_\varphi(t):=\int \varphi~d\mu_t^\e, \qquad g_\varphi(t):=\int \varphi~d\mu_t. \]  
Let $C_\varphi$ be as in Remark \ref{rem_finite}, then since the estimates are uniform in $\e$ we may let $\e_n \rightarrow 0$ in that Remark to obtain   
\begin{equation}\label{eq:lipschitz_varphi} |g_\varphi(t_k) -g_\varphi(t^\ast_k)| \lesssim C_\varphi |t_k-t^\ast_k| \cdot C(D). 
\end{equation} 
Letting $k$ tend to $\infty$ we obtain  
\[ \int_{\mathbb R^D} \varphi ~d\nu =\int_{\mathbb R^D} \varphi ~d\nu^\ast.\] 
As $\varphi \in C_c^\infty(\mathbb R^D)$ is arbitrary we conclude that $\nu=\nu^\ast$ which proves the claim.

{\bf 2.} Claim: If $t \in (0,T)$ then $(\mu_t^{\e_n})$ converges in the distributional sense to $\mu_t \in \mathscr M.$

{\it \hskip 0.2in Proof of claim.} We use Equation (\ref{eq:lipschitz_varphi}) and the way $t \rightarrow \mu_t$ has been extended to $(0,T)$ to obtain for $t , t^* \in (0, T)$
\begin{equation}\label{eq:lipschitz_varphi2} 
|g_\varphi(t) -g_\varphi(t^\ast)| \lesssim \bar C |t-t^\ast| \qquad \forall t, t^\ast \in (0,T), \qquad \bar C:= C_\varphi \cdot C(D). 
\end{equation} 
Fix $t \in (0,T)$. Then for $t_k \in \mathscr D$ we have 
\[\begin{split}|g^{\e_n}_\varphi(t)-g_\varphi(t)| &\leq  |g^{\e_n}_\varphi(t)-g^{\e_n}_\varphi(t_k)|+ |g^{\e_n}_\varphi(t_k)-g_\varphi(t_k)|+ |g_\varphi(t_k)-g_\varphi(t)|\\
&\lesssim \bar C |t-t_k|+|g^{\e_n}_\varphi(t_k)-g_\varphi(t_k)|
. \end{split}\]
We first let $n$ tend to $\infty$ and then $t_k$ tend to $t$ to conclude that 
\[\lim_{n \rightarrow \infty} g^{\e_n}_\varphi(t)=g_\varphi(t)\] 
which proves the claim.

\begin{remark}
We remark that it is in fact also possible to abstractly retrieve some limiting velocity fields $(v_t)_{t \in [0, T]}$ so that together with the limiting measures $(\mu_t)_{t \in [0, T]}$ the continuity equation 
\begin{equation}\label{cont_eq} \partial_t \mu_t + \nabla \cdot (\mu_t v_t) = 0\end{equation}
is satisfied.  Indeed, the basis for such a result is Lemma 7.2 in \cite{wilfrid_ambrosio}, which can be adapted to deficient measures \emph{mutatis mutantis}.  In our case we also have the additional complication that the velocity fields have only a \emph{local} bound.  However, this can be dispensed with by inserting another diagonalization procedure where we consider finite volume measures, $\mu_t^{\e, L}$, which are supported in regions of scale $L$.  Then we may first take $\e$ to zero, and then $L$ to infinity.  For the principal results of this note, this route will be avoided due (in part) to the fact that the velocity field and its limit must (and will) be produced on the basis of an explicit dynamical structure.
\end{remark}

\end{proof}

%
%
%
%
%
%
%
%
\subsection{Hamiltonian Dynamics With Mass Dissipation}\label{e_reg_exist}
\begin{defn}\label{defn_ham}
Given $\mu$ a measure on $\mathbb R^D$ where $D = 2d$ -- and where we denote $x = (p,q)$ -- we define our Hamiltonian to be 
\begin{equation}\label{eq_ham} \mathscr H(\mu) = \frac{1}{2} \int |p|^2~d\mu(x) + \frac{1}{2}\int (W * \mu)(q)~d\mu(x) + \int \Phi(q)~d\mu(x), \end{equation}
where $W$ and $\Phi$ are both functions of $q$ and $W$ is even.  We further assume that: 
\begin{itemize} 
\item[$\circ$] $W \in C_c^2(\mathbb R^d)$.
\item[$\circ$] $\Phi \in C^2(\mathbb R^d)$.
\end{itemize} 
We let $a$ denote the range of the interaction -- i.e., $W$ is supported on a ball of radius $a$.  Although not always strictly necessary, we shall further assume, with no essential loss of generality, that $\Phi$ is polynomially bounded:
$$
|\Phi(q)| \leq B_{1}|q|^{b_2}
$$
for finite constants $B_{1}$ and $b_{2}$.

Formally, the gradient of $\mathscr H$ with respect to the $2$--Wasserstein  metric at $\mu$ is the functional 
\begin{equation}\label{grad_ham} 
 (p,q) \rightarrow \nabla_{W} \mathscr H(\mu)(p,q) =[p, \nabla (W \ast \mu + \Phi)(q)]  =:V_\mu(p,q)
 \end{equation} 
provided that $\mu$ is sufficiently well--behaved (cfr., e.g., \cite{AGN} or \cite{villani_old}).  We shall use Equation \eqref{grad_ham} in order to define the relevant dynamics.   

Let $J$ be the $D \times D$ symplectic matrix so that $J(p,q)=(-q,p)$.  We say that $(\mu_t^\e)_{t \in [0, T]}$ solves the {\it deficient Hamiltonian ODE} with initial condition $\mu_0$ if it satisfies 
\begin{equation}\label{dhamode} \partial_t \mu_t^\e + \nabla \cdot (J \nabla_W \mathscr H(\mu_t^\e) \mu_t^\e) = -\varepsilon |J \nabla_W \mathscr H(\mu_t^\e)|\mu_t^\e. 
\end{equation}
Similarly, we say that $(\mu_t)_{t \in [0, T]}$ solves the {\it Hamiltonian ODE} with initial condition $\mu_0$ if it satisfies 
\begin{equation}\label{dhamode_limit} \partial_t \mu_t + \nabla \cdot (J \nabla_W \mathscr H(\mu_t) \mu_t) = 0 .\end{equation}
Equations~\eqref{dhamode} and \eqref{dhamode_limit} are again understood in the appropriate distributional sense, by testing against functions $\varphi \in C_c^\infty((0,T) \times \mathbb R^D)$.  
\end{defn}

To enable us to extract limiting Hamiltonian dynamics, let us now prove: 
%
\begin{lemma}\label{lemma_ham}  Set $V_\mu= J \nabla_{W} \mathscr H(\mu)$ and suppose  
\begin{itemize}  
\item[$\circ$] $(\mu_n)$  is of uniformly bounded total mass and converges to $\mu$ in the distributional sense.
\item[$\circ$] We have the tightness in $p$ condition: $\lim_{r \rightarrow \infty}C_r(\bar q)=0$ for all $\bar q$ where
\[ C_r(\bar q)=\sup_{n \in \mathbb N} \int_{ B^c_r \times \mathbb R^d} |\nabla W|(\bar q -q) d\mu^n(p,q).\] 
\end{itemize} 
Then  $(V_{\mu_n})$ converges uniformly to $V_\mu$ on compact sets.  So, in particular, 
\[ V_{\mu_n}\mu_n \rightharpoonup V_\mu\mu \quad \hbox{and} \quad |V_{\mu_n}| \mu_n \rightharpoonup |V_\mu| \mu\]
in the sense of distribution. 
\end{lemma}

\begin{proof} It suffices to show that $(\nabla W * \mu_n)$ converges uniformly to $\nabla W * \mu$ on compact sets. As $(\mu_n)$  is of uniformly bounded total mass, $(\nabla W * \mu_n)$ is a bounded subset of $W^{1,\infty}(\mathbb R^d)$  
and so,  by the Ascoli-Arzela theorem $(\nabla W * \mu_n)$ is precompact in 
$C(B_Q)$
-- where $B_{Q}$ is a ball of radius $Q$ in $\mathbb R^{d}$ --
 for any $Q>0$. To show that $(\nabla W * \mu_n)$ converges uniformly to $\nabla W * \mu$ on $B_Q$ it suffices to show that it converges pointwise  to $\nabla W * \mu$. 

For $r>0$ let $\theta_r \in C(\mathbb R)$ be a monotone nondecreasing continuous function such that $0 \leq \theta_r \leq 1$, 
\[\begin{cases} \theta_r(p) = 1~~~\mbox{for $|p| \geq r$}\\
\theta_r(p) = 0~~~\mbox{for $|p| \leq r -1$}.  
\end{cases}\]
Using for $\mu_n$ the decomposition  
\[ \nabla W * \mu_n(\bar q)= \int_{\mathbb R^{D}}  \nabla W(\bar q - q) (1-\theta_r(|p|))~d\mu_n(p,q)   + 
 \int_{\mathbb R^{D}}  \nabla W(\bar q - q) \theta_r(|p|)~d\mu_n(p,q)
\]
and writing a similar decompostion for $\mu$ we obtain 
\begin{equation}\label{e:vel_convergence1}
 |\nabla W * \mu_n(\bar q)-\nabla W * \mu(\bar q)|  \leq  
|\int_{\mathbb R^{D}}  \nabla W(\bar q - q) (1-\theta_r(|p|))\bigl(d\mu_n(p,q)-d\mu(p,q)\bigr)|
 + 2 C_{r-1}(\bar q)
\end{equation} 
To obtain the upper bound in Equation (\ref{e:vel_convergence1}) we have used distributional convergence: \[  \int_{\mathbb R^D}  |\nabla W(\bar q - q)| \theta_r(|p|)~d\mu(p,q) \leq 
 \limsup_{n\rightarrow \infty} \int_{\mathbb R^D}  |\nabla W(\bar q - q)| \theta_r(|p|)~d\mu_n(p,q) \leq  C_{r-1} (\bar q).  \] 
 Finally, we use that $(p,q) \rightarrow \nabla W(\bar q - q) (1-\theta_r(|p|))$ is of compact support and again utilize the fact that $(\mu_n)$  converges to $\mu$ in the distributional sense to conclude from Equation (\ref{e:vel_convergence1}) that 
 \[  \limsup_{n\rightarrow \infty}  |\nabla W * \mu_n(\bar q)-\nabla W * \mu(\bar q)|  \leq 2 C_{r-1}(\bar q). \]
Letting $r$ tend to $\infty$ we have that $(\nabla W * \mu_n(\bar q))$ converges  pointwise  to $\nabla W * \mu(\bar q)$.
\end{proof}
%
\begin{thm}[Existence of Solution to Deficient Hamiltonian ODE]\label{exist_def_ODE}
For fixed $\alpha_0 > 0$, $\varepsilon > 0$, and $0 < T < \infty$, let $\mu_0 \in \mathcal M_{\alpha_0}$ denote some initial Borel measure on $\mathbb R^D$ and let $\mathscr H$ be the Hamiltonian in Definition \ref{defn_ham}.  Assume for simplicity that the total mass of $\mu_0$ is $1.$ Then there exists a path $t \rightarrow \mu_t^\varepsilon \in \mathcal M_{\alpha}$ where $\alpha < \min\{\alpha_0, \varepsilon\}$ such that 
\begin{itemize}
\item[(i)] $(\mu_t^\varepsilon)_{t \in [0, T]}$ satisfies the deficient Hamiltonian ODE 
\begin{equation} \label{def_eq} 
\partial_t \mu_t^\varepsilon + \nabla \cdot (J \nabla_{W} \mathscr H(\mu_t^\varepsilon)\mu_t^\varepsilon ) = -\varepsilon |J \nabla_{W} \mathscr H(\mu_t^\varepsilon)|\mu_t^\varepsilon. 
\end{equation}  

\item[(ii)] $t \rightarrow \mu_t^\varepsilon \in \mathcal M_\alpha$ is narrowly continuous and $M_{\alpha}(\mu_t^\e)$ is monotonically nonincreasing in $t$.  In particular, $M_0(\mu_t^\e) \leq 1$ for $t \in [0, T]$.
\end{itemize}
\end{thm}

\begin{proof}
The construction of $\mu^\varepsilon_t$ uses roughly the discretization scheme of $\S 6$ of \cite{wilfrid_ambrosio}, which goes as follows: 
\begin{enumerate}
\item For $n \geq 1$ define the step size $h = T/n$.  
\item We start with $\mu_0^{\e, n} = \mu_0$ and define $v_0^{\e, n} = J \nabla_W \mathscr H(\mu_0^{\e, n})$. 
\item For $t \in [kh, (k+1)h)$ we define $\mu_t^{\e, n}$ to be the solution to the deficient equation given in Proposition \ref{preserv_exp} with the constant velocity field
\[ v_{kh}^{\e, n} = J \nabla_W \mathscr H(\mu_{kh}^{\e, n}).\]
\end{enumerate}
By construction, we therefore see that $(\mu_t^{\e, n})_{t \in [0, T]}$ satisfies 
\begin{equation}\label{disc_ham_eqn} 
\partial_t \mu_t^{\e, n} + \nabla \cdot\left( \mu_t^{\e, n} J\nabla_W \mathscr H(\mu_{[t/h]h}^{\e, n})\right) = - \e \left|J\nabla_W \mathscr H(\mu_{[t/h]h}^{\e, n})\right| \mu_t^{\e, n}.
\end{equation}
Furthermore, Proposition \ref{preserv_exp} allows us to write 
\begin{equation}\label{disc_ham_represent} 
\int_{\mathbb R^D} \varphi~d\mu_t^{\e, n}= \int_{\mathbb S_t^{\e, n}} (R_t^{\e, n} \varphi) \circ X_t^{\e, n} d\mu_0
\end{equation} 
for all $\varphi \in C_c^\infty(\mathbb R^D).$ Here $X^{\e, n}$ is the flow defined by 
\[ \dot X^{\e, n}_t= v^{\e, n}_t (X^{\e, n}_t), \qquad X^{\e, n}_0=x\] 
and 
\[ R_t^{\e, n} \circ X^{\e, n}_t= \exp\Bigl(-\int_0^t \varepsilon | v^{\e, n}_s ( X_s^{\e, n})|~ds \Bigr).\]
%
By Proposition \ref{preserv_exp}, (iii) $t \rightarrow M_\alpha(\mu_{t}^{\e, {n}})$ is monotone nonincreasing. In particular,  the total masses of the $\mu_t^{\e, n_k}$'s are uniformly bounded. Proceeding as in (the proof of) Proposition \ref{narrow_conv} we obtain the existence of an increasing sequence of natural numbers $(n_k)_{k \in \mathbb N}$ such that as $k$ tends to $\infty$,  $(\mu_t^{\e, {n_k}})_{k \in \mathbb N}$ converges in the sense of distribution to a measure $\mu_t^{\e}$ for each $t \in [0,T].$ In order to avoid adding a new subscript we shall write that $(\mu_t^{\e, {n}})_{n \in \mathbb N}$ converges in the sense of distribution to a measure $\mu_t^{\e}$ for each $t \in [0,T]$ where $n$ is restricted to an appropriate subset of $\mathbb N.$ Observing that by disintegration and Markov's inequality that for all $r>0$ and $\bar q \in \mathbb R^d$ we have
\begin{equation}\label{est_a_priori}\int_{B_r^c \times \mathbb R^d}|\nabla W(\bar q-q)|~ d\mu_t^{\e, {n}}(p,q) \leq e^{-\alpha r} \int e^{\alpha |(p,q)|} |\nabla W(\bar q-q)| ~d\mu_t^{\e, {n}}(p,q) \leq e^{-\alpha r} M_{\alpha}(\mu_0) \|W\|_{C^2} ,
 \end{equation}
we can employ Lemma \ref{lemma_ham} to obtain for fixed $t \in [0,T]$ 
\begin{equation}\label{e:con_momentum1}
\mu_t^{\e, {n}} \nabla_W \mathscr H(\mu_t^{\e, {n}}) \quad \hbox{converges in the distributional sense to} \quad \mu_t^{\e} \nabla_W \mathscr H(\mu_t^{\e})
\end{equation} 
and 
\begin{equation}\label{e:con_momentum2}
\mu_t^{\e, {n}}| \nabla_W \mathscr H(\mu_t^{\e, {n}})| \quad \hbox{converges in the distributional sense to} \quad \mu_t^{\e} | \nabla_W \mathscr H(\mu_t^{\e})|.
\end{equation} 


By Remark \ref{re:bounded_moment}, (iii)  there exists a constant $\bar m$ independent of $t \in [0,T]$ and $k \in \mathbb N$ such that 
\[| \nabla_W \mathscr H(\mu_{[t/h]h}^{\e, {n}})- \nabla_W \mathscr H(\mu_t^{\e, {n}})| =| \nabla W \ast \mu_{[t/h]h}^{\e, n} - \nabla W \ast \mu_{t}^{\e, n}| \leq h \bar m. \]
This, together with Equations (\ref{e:con_momentum1}), (\ref{e:con_momentum2}) and the fact that the total masses of $\{\mu_t^{\e, {n}}\}_{t, n}$ are bounded uniformly in $t$ and $n$ imply that 
\begin{equation}\label{e:con_momentum1b}
\mu_{t}^{\e, {n}} \nabla_W \mathscr H(\mu_{[t/h]h}^{\e, {n}}) \quad \hbox{converges in the distributional sense to} \quad \mu_t^{\e} \nabla_W \mathscr H(\mu_t^{\e})
\end{equation} 
and 
\begin{equation}\label{e:con_momentum2b}
\mu_t^{\e, {n}}| \nabla_W \mathscr H(\mu_{[t/h]h}^{\e, {n}})| \quad \hbox{converges in the distributional sense to} \quad \mu_t^{\e} | \nabla_W \mathscr H(\mu_t^{\e})|.
\end{equation} 
We combine Equations (\ref{disc_ham_eqn}), (\ref{e:con_momentum1b}) and (\ref{e:con_momentum2b}) and use the fact that $(\nabla_W \mathscr H(\mu_t^{\e, {n}}) )$ is bounded uniformly in $t, n$ on compact sets to conclude that  $(\mu^\e_t)_{t \in [0, T]}$ satisfies Equation (\ref{def_eq}).  Thus reasoning as in the proof of Proposition \ref{preserv_exp}, $t \rightarrow M_\alpha(\mu_{t}^{\e})$ is monotone nonincreasing.  The narrow continuity claim of item (ii) now immediately follows from Lemma \ref{lemma_cont_def_cont}.
\end{proof}

\begin{remark}
We note that while in order to obtain existence of dynamics with non--zero $\e$ an \emph{a priori} estimate as in Equation \eqref{est_a_priori} already suffices, more is required to obtain some control which is uniform in $\e$ to retrieve limiting ($\e = 0$) dynamics.  Here is where the dynamical considerations will come into play (in particular, see Lemma \ref{lem_tightness}).  
\end{remark}


%
%
%
%
%
%
%
%
\section{Dynamical Hypothesis}\label{d}

Here we will let $(p, q)$=(position, momentum) denote canonical variables and $p_t, q_t$ denote the associated Lagrangian trajectories (or characteristics) with dynamics dictated by the relevant Hamiltonian.  Indeed, we shall have occasion to consider single particle Hamiltonian dynamics with some Hamiltonian $H$ (the Hamiltonian $\mathscr H$ as defined in Definition \ref{defn_ham} is the integrated total Hamiltonian of the whole system).  We recall the cannonical equations of Hamiltonian dynamics: 
\begin{equation}\label{symp_sys} \dot p = - \frac{\partial H}{\partial q},~~~~~
\dot q = \frac{\partial H}{\partial p}.
\end{equation} 
We will assume that the Hamiltonian is given by 
\[H(p, q,t) = \frac{1}{2} |p|^2 + \Phi(q)+ \Psi(t,q).\] 
Here, $\Phi, \Psi(\cdot,t) \in C^2(\mathbb R^d),$ $\Psi$ is a Borel function defined on $[0,\infty) \times \mathbb R^d$ and there exists $B>0$ such that  
\begin{equation}\label{condition-on-Lip(Psi)} 
|\nabla \Psi(t,q)| < B \quad \forall \; t \geq 0, \; q \in \mathbb R^d.
\end{equation} 
Let $u \in C^2(\mathbb R)$ be such that 
\begin{equation}\label{condition-on-Lip(u)} 
u'(r) \geq  B+ \max_{|q|=r} \langle \nabla \Phi(q), \hat q\rangle, \qquad r \geq 0,  
\end{equation}
where $\hat q|q|=q.$ We consider the auxiliary ``Hamiltonian''  
\[\tilde H(p,q)={1 \over 2} \langle p, \hat q \rangle^2+ \Upsilon(q) \quad \hbox{where} \quad  \Upsilon(q)=u(|q|).\]

%
\begin{lemma}[Single Particle Dynamics]\label{rad_outward}
Consider single particle Hamiltonian dynamics Equation (\ref{symp_sys}). 
We define a $\star$--ring by the condition that 
\begin{equation}\label{e:assumption-on-Up} \Upsilon(q) < \Upsilon(q^\star), ~~~\mbox{for all $|q| > |q^\star|$} 
\end{equation}
and assume that $q^\star \not =0.$ 
Fix $\bar t \in [0,T]$. Then for all characteristics which start out  inside  the region bounded by the $\star$--ring, in the sense that $|q_0|<|q^\star|$, either 
\[ |q_t|  \leq |q^\star|, ~~~\mbox{for all $t \geq \bar t$,}\]
or
\[ |q_t| \rightarrow \infty\]
with nonvanishing radial speed, i.e., there exists some some $t_*>\bar t$ such that $|q_{t_*}| = |q^\star|$ and 
\[ \frac{d|q_t|}{dt}(t) \geq \frac{d|q_t|}{dt}(t_\ast) > 0, ~~~\mbox{for all $t \geq t_* $.}\]
\end{lemma}
\begin{proof} 
We first claim that, provided $\frac{d|q|}{dt} > 0$ on an interval, the quantity $\tilde H$ 
is increasing on that interval.  Using Equations (\ref{condition-on-Lip(Psi)}) and (\ref{condition-on-Lip(u)}), direct computations give 
\begin{equation}\label{e:second-der-of-|q|}
{d^2 |q| \over dt^2} ={ \langle \ddot q , q \rangle \over |q|} + { |\dot q|^2 |q|^2- \langle \dot q , q \rangle^2 \over |q|^3} \geq  { \langle \ddot q , q \rangle \over |q|} = -\langle \nabla \Phi(q) +\nabla \Psi(t,q), \hat q \rangle> -\langle \nabla \Upsilon(q), \hat q \rangle 
\end{equation}
where the last inequality is strict due to (the strict inequality in) Equation (\ref{condition-on-Lip(Psi)}).

As $\Upsilon$ depends only on $|q|$,
\begin{equation}\label{e:second-der-of-|q|2}
\langle \nabla \Upsilon(q) , \dot q \rangle = \langle \nabla \Upsilon(q) , \hat q \rangle \frac{d|q_t|}{dt} 
\end{equation}
and so, by  Equation \eqref{e:second-der-of-|q|}  
\[\frac{d\tilde H}{dt} = {d |q| \over dt} {d^2 |q| \over dt^2} +\langle \nabla \Upsilon(q) , \dot q \rangle=  {d |q| \over dt} 
\Bigl( {d^2 |q| \over dt^2} + \langle \nabla \Upsilon(q) , \hat q \rangle \Bigr)> 0.\] 
Suppose now that it is not the case that $|q_t| \leq |q^\star|$ for all $t \geq \bar t$. Set 
$$
t_\ast=\inf \{t : |q(t)|>|q^\star|, \, t \geq \bar t  \}.
$$ 

\textbf{1.}  Claim: We have 
\[|q(t_\ast)|=|q^\star| \quad \hbox{ and} \quad  {d|q| \over dt}(t_\ast)> 0 
\] 

{\it \hskip 0.2in Proof of claim.} What is obvious is that $|q(t_\ast)|=|q^\star|$ and the derivative of $|q|$ at $t_\ast$ is nonnegative.  Assume on the contrary that the derivative vanishes. Then we necessarily have (by e.g., simple expansion) that  the second derivative of $|q|$ at $t_\ast$ is nonpositive and so since Equation (\ref{e:second-der-of-|q|}) reads 
\begin{equation}\label{e:second-der-of-|q|3} 
{d^2|q|\over dt^2} (t_\ast) > -u'(|q^\star|) 
\end{equation} 
whereas Equation (\ref{e:assumption-on-Up}) yields $u'(|q^\star|) \leq 0$ we have a  contradictiton.  

\textbf{2.}  Claim:  $\frac{d|q|}{dt}(t) \geq \frac{d|q|}{dt}(t_*)$ for all $t \geq t_*$.  

{\it \hskip 0.2in Proof of claim.} Assume on the contrary that 
\[E=\Bigl\{ t \geq t_\ast:  {d |q|\over dt}  (t) < {d |q|\over dt}  (t_\ast) \Bigr\} \neq \emptyset. \] 
Let $t_1$ be the infimum of $E.$ First it is noted that $t_1 > t_*$ since $\frac{d^2|q|}{dt^2}(t_*) > 0$ by Equation \eqref{e:second-der-of-|q|3}.  But then 
$$
{d |q|\over dt}  (t_1) = {d |q|\over dt}  (t_\ast) < {d |q|\over dt}  (t) \quad \forall t \in (t_\ast,t_1)
$$
and so, $|q(t_1)|>|q^\star|.$ Thus, 
\[\tilde H(p(t_1),q(t_1)) < \tilde H(p(t_*), q(t_*)) \]
which is a contradiction, since $\tilde H$ is \textit{increasing} on the interval $[t_*, t_1]$.
\end{proof}

\begin{remark}\label{re:rad_outward} Note that in the proof of Lemma \ref{rad_outward} we have proven, in the last claim, the following general result. If $t_\ast \in [0,T)$ satisfies 
\[|q(t_\ast)|=|q^\star|>0 \quad \hbox{ and} \quad  {d|q| \over dt}(t_\ast)> 0 \] 
then $\frac{d|q|}{dt}(t) \geq \frac{d|q|}{dt}(t_*)$ for all $t \geq t_*$. In particular  $\tilde H$ is increasing on the interval $[t_*, T]$.
\end{remark}

We make final assumptions on $u'$ which will allow us to summarize our observations in this section in the following dynamical hypothesis to which we shall refer later:
\begin{hyp}\label{dyn_hyp} 
We postulate existence of $\Upsilon(q)=u(|q|)$ where $u \in C^2(\mathbb R)$ satisfies Equation (\ref{condition-on-Lip(u)}). We assume that there is a sequence $\{q_L^\star\}_{L=1}^\infty$ such that $ |q_L^\star|$ increases to $\infty$ and for all $q_{L}^{\star}$,
\begin{equation}\label{e:increasing_points}
\Upsilon (q)< \Upsilon(q_L^\star) \quad \hbox{for all } \quad |q|> |q^\star_L|.
\end{equation}
As a matter of notation we will sometimes denote generic elements of the sequence by $L^\star$ or just $L$ when the context is clear.

Let us also define the phase space \emph{cylinders} 
\[ \Omega_{L^\star} = \mathbb R^d \times B_{L^\star},\]
where $B_{L}$ denotes the ball in $\mathbb R^{d}$ of radius $L$ centered at the origin,
and refer to them as \emph{spatial} regions of no return.
\end{hyp}

%
\begin{remark}
\label{DSN}
We note that the existence of unbounded (in $q$) spatial regions of no return is certainly guaranteed by the condition that $\Upsilon(q)$ decreases in $|q|$.  Moreover, if additionally $\Upsilon(q) \sim -|q|^{1+R}$ with $R > 1$, then, inevitably, the unbounded motion will reach infinity in finite time.  \end{remark}

\begin{remark}\label{NY}
We emphasize that since the dynamical hypothesis is some uniform in $\e$ control on the dynamics, the no return condition is inherited in any $\e \rightarrow 0$ limit.   
\end{remark}


%
%
%
%

\section{Consequences of Dynamical Hypothesis}\label{sec_mass_conv}

\subsection{Limiting Hamiltonian ODE}\label{sec_defn-set}

%

As a first consequence of the dynamical hypothesis we will retrieve some limiting dynamics (meaning the relevant equation of continuity driven by the appropriate velocity field).  In light of the content of Lemma \ref{lemma_ham}, we see that we are first to estimate, for fixed $t$, the quantity
$$
\tilde C_r(\mu^{\e}_t,\bar q) := \int_{\mathbb R^D}  \theta_r(p)|\nabla W(\bar q-q)|~d\mu^{\e}_t(p,q),
$$
(in the ensuing we will omit the tilde as it should cause no confusion) where $\mu_t^\e$'s are given by Theorem \ref{exist_def_ODE} and $\theta_r$ is supported outside the ball of radius $r$ (Lemma \ref{lem_tightness}).  Further, we will have to produce some control on the time evolution of the relevant velocity fields, which is now not ``automatic'' since $\e$ is tending to zero (Lemma \ref{le:bound-on-times1}).

Since the up and coming argument requires pulling trajectories back to $\mu_0$, we shall work directly with the time discretized measures which, by construction, satisfy the pushforward equation $\mu_t^{\e, n} = X_t^{\e, n} \# \mu_0$ (see Proposition \ref{preserv_exp} as it applies in the proof of Theorem \ref{exist_def_ODE}).  

%
%


\bigskip

We have the following tightness estimate:

\begin{lemma}\label{lem_tightness}
Let $T > 0$ and let $(\mu_t^{\e, n})_{t \in [0, T]}$ be the time discretized measures as constructed in the proof of Theorem \ref{exist_def_ODE}.   Suppose further that 
\begin{itemize}
\item[$\circ$] $W$ is supported on the ball of radius $a$ around the origin and there is some $B \geq 0$ such that 
\[ |W| \leq B ~~~\mbox{and}~~~ |\nabla W| < B;\]
\item[$\circ$] There is a ``bounding'' potential $\Upsilon$, which is uniform in $\e, n$, corresponding to $\Phi$ and $\Psi^n := W * \mu_t^{\e, n}$ (see Equation \eqref{condition-on-Lip(u)} and the display which follows) satisfying the condition in Hypothesis \ref{dyn_hyp}.  
\end{itemize}

Let $Q > 0$ and let $C_r(\mu_t^{\e, n}, \overline q)$ be defined as in Lemma \ref{lemma_ham}, then for $L$ as in Hypothesis \ref{dyn_hyp} sufficiently large so that (cfr., Definition \ref{defn_ham} for the meaning of $a$) \[ |q_L^\star| > Q + a,\]
we have 
\[ \forall \overline q \in B_Q(0), ~~~\forall t \in [0, T], ~~~\forall \e > 0,~~~\forall n \]
the bound 
\[ C_r(\mu_t^{\e, n}, \overline q) \leq B\left(\mu_0(Q_L^\star) + \mu_0(B_{r^*}^c \times \mathbb R^d)\right).\]
Here $Q_L^\star$ denotes the complement of the cylinder $\Omega_L^\star: = \mathbb R^d \times B_{|q_L^\star|}$ and $r$ is sufficiently large so that (at least)
\[r^* = r-1 - M_L T > 0, \]
where 
\[ M_L = B + \sup_{q \in B_{|q_L^\star|}} |\nabla \Phi(q)|.\]
\end{lemma}
\begin{proof}
We have that 
\[\begin{split} C_r(\mu_t^{\e, n}, \overline q) &= \int |\nabla W(\overline q - q_t)| \theta_r(p_t)~d\mu_t^{\e, n}\\
&= \int \theta_r(p_t) |\nabla W(\overline q - q_t)| R_t^{\e, n}(p_t, q_t)~d\mu_0(p, q);
\end{split}\]
here it is re--emphasized that the validity of the pull back to $\mu_0$ has been assured by the manner in which the measures $\mu_t^{\e, n}$ have been constructed.  By invoking the dynamical condition, for trajectories starting inside $\Omega_L^\star$ this quantity can be bounded depending on whether the position marginal of the trajectory has ever left $B_{|q_L^\star|}$ by time $t$: If the trajectory never left, then the acceleration can be bounded by $M_L$ whereas if the trajectory leaves, then it is guaranteed not to return by Lemma \ref{rad_outward} and so since $|q_L^\star|$ is outside the interaction range of any point in $B_Q$, $\nabla W(\overline q - q_t) = 0$.  

More precisely, let us partition the space of all possible \emph{initial} conditions in phase space into three sets: 
\[\begin{split}\mathcal S &= \{(p, q) \in \mathbb R^D \mid |q_s| \leq |q_L^\star|~~~\forall s \in [0, t]\}\\
\mathcal G &= \{(p, q) \in \mathbb R^D \mid |q| < |q_L^\star| \mbox{~~~and $\exists \overline t \in (0, t)$ such that $|q_{\overline t}| > |q_L^\star|$}\}\\
\mathcal O &= \{ (p, q) \in \mathbb R^D \mid |q| \geq |q_L^\star|\}.\end{split}\]
Since $|q_t| \geq |q_L^\star| > Q + a$ for $q \in \mathcal G$ so that $\nabla W(\overline q - q_t) \equiv 0$ there, it is the case that we have 
\begin{equation}\label{penult}\begin{split} C_r(\mu_t^{\e, n}, \overline q) &= \int_{\mathcal S \cup \mathcal O} \theta_r(p_t) |\nabla W(\overline q - q_t)|R_t(p_t, q_t)~d\mu_0(p, q)\\
&\leq B \mu_0(Q_L^\star) + B\int_\mathcal S \theta_r(p_t)~d\mu_0(p, q).
\end{split}\end{equation}
Here in the last inequality we have used that $|\nabla W| \leq B$ and $|R_t| \leq 1$.

Since all measures under consideration have mass bounded by one, it is immediate that for all $t \in [0, T]$,
\[ |\nabla W * \mu_t^{\e, n}| < B,~~~|W*\mu_t^{\e, n}| \leq B, \]
Therefore, for $(p, q) \in \mathcal S$, we have that by construction of $\Upsilon$, 
\[ \left|\frac{dp_s}{dt}\right| = \left|\nabla (\Phi + W *\mu_t^{\e, n})(q_s)\right| \leq M_L, ~~~\forall s \in [0, t], \]
from which it directly follows that 
\[ |p_t| \leq M_L t + |p_0|.\]
On the other hand, for the trajectory to contribute to the last integral in \eqref{penult} we must have $\theta_r (p_t) > 0$, so altogether we have 
\[ r - 1 \leq |p_t| \leq M_L t + |p_0|, \]
so that 
\[ |p_0| \geq r - 1 - M_L t = r^*,\]
yielding the conclusion.
\end{proof}
\begin{remark}\label{e_tight}
Let us fix $t \in [0,T]$. 
\begin{enumerate}[(i)]
\item In reference to the above lemma, for all $t \in [0, T]$ the entire bound now resides with $\mu_0$ which is a particular (finite) measure.  Thus we may choose $L^\star$ large so that the first term $\mu_0(Q_L^\star)$ is small and then, by choosing $r$ large, the second term can be made small.  Not only is this uniform in $t$, we further note that the estimates of Lemma \ref{lem_tightness} are uniform in the discretization $n$ and also $\e$, as the dynamical condition uniformly bounds the dynamics. 

\item Let us set 
\[ F^{\e, n}_t:= \nabla W * \mu^{\e,n}_t, \quad F^\e_t = \nabla W * \mu_t^\e\]
and suppose that $\mu_t^{\e, n} \rightharpoonup \mu_t^\e$.  We record that the above Lemma in particular implies that the hypothesis of Lemma \ref{lemma_ham} are satisfied and hence $F_t^{\e, n} \rightarrow F_t^\e$ uniformly on compact sets.  (Recall that we have already established this via different means in the proof of Theorem \ref{exist_def_ODE}).  
\end{enumerate}
\end{remark}

Next we will acquire the required control on the time evolution of the (Hamiltonian) velocity fields.  
Let us denote by
\[v^{\varepsilon,n}_t(p,q)=\bigl[ -\nabla (\Phi+ W \ast \mu^{\varepsilon,n}_{[{ t \over h_n}]h_n}  )(q),p \bigr]:= [-\nabla \Phi(q) + F^{\e, n}_{t_n}(q), p] \]
(where by slight abuse of notation, $t_n= t_n(t)$ is the nearest time discretization point) the relevant velocity field for the time discretized measures $\mu_t^{\e, n}$.

Suppose that $M_\alpha(\mu_0) < \infty$ for some $\alpha > 0$. Formally, from the deficient equation of continuity, we have 
\[ 
\partial_t F^{\varepsilon, n}_t(\bar x)= \int_{\mathbb R^D}\bigl(p\cdot \nabla^2 W(\bar q-q)  -\varepsilon |v_t^{\varepsilon, n}(x)|  \nabla W(\bar q-q) \bigr) ~d \mu_{t}^{\varepsilon, n}(x)=: A^{\varepsilon, n}_{t,2}-\varepsilon A^{\varepsilon, n}_{t,1},
\] 
where the above $\nabla^2$ is notation for the matrix of second derivatives.
This requires justification since, strictly speaking, $W$ is not compactly supported in phase space.  However, since $\e > 0$, we have exponential moments for $\mu_t^{\e, n}$ (see Proposition \ref{preserv_exp}, (iii); here, in hindsight we may regard the $v_t^{\e, n}$ used in the construction of $\mu_t^{\e, n}$ as prescribed) and all relevant functions are $C^1$ so item (iii) of Remark \ref{re:bounded_moment} can be applied.

Recalling once again that for $n < \infty$ we may pullback to the initial measure, we rewrite the above as
\[A^{\varepsilon, n}_{t,2}=\int_{\mathbb R^D}p_t\cdot \nabla^2 W(\bar q-q_t)  R^{\varepsilon, n}_t(p_t, q_t) ~d\mu_0(p,q)  \]
and 
\[A^{\varepsilon, n}_{t,1} = \int_{\mathbb R^D} |v_t^{\varepsilon,n}(p_t,q_t)|  \nabla W(\bar q-q_t) R^{\varepsilon,n}_t(p_t, q_t) ~d\mu_0(p,q). \]
Let $\{q^\star_L\}_{L=1}^\infty$ be the sequence of no--return points as  in Hypothesis \ref{dyn_hyp}. We set $q^\star_0=0.$ Given $q \in \mathbb R^d$,  there exists a unique $L$ such that $|q^\star_{L}| \leq |q| < |q^\star_{L + 1}|.$ Let us define  
\[\mathcal R(q)= B+ \max_{|q'| \leq |q^\star_{L+1}|} |\nabla \Phi(q')| \qquad \forall \;\; |q^\star_L| \leq |q| < |q^\star_{L+1}|. \]   
For the potentials that we have in mind, under the assumption of an exponential moment for $\mu_{0}$, we certainly have that 
\begin{equation}\label{eq:0}
\int_{\mathbb R^{2d}} \mathcal R(q)~d\mu_0< \infty.
\end{equation}

%
%
\begin{lemma} \label{le:bound-on-times1} Let $Q>|q^\star_1|$ and let $D_Q \subset \mathbb R^{D}$ be the ball of radius $Q$, centered at the origin.  Let $\mu_t^{\e, n}$, $v_t^{\e, n}$ and $\mu_0$ be as described and suppose that Equation (\ref{eq:0}) holds.  (In particular, for the polynomially bounded potentials as discussed in Definition \ref{defn_ham}, and for $\mu_{0}$ with an exponential moment, this is the case.)  Then there exists a constant $C_Q$ independent of $\e, n$ such that  
\begin{equation}\label{vb1}\sup_{\bar x,t }  \{ |\nabla F^{\varepsilon, n}_t(\bar x)| \; | \;\bar x \in D_ Q, t \in [0,T] \}\leq C_{Q}  \end{equation}
and  
\begin{equation}\label{vb2}\sup_{\bar x,t }  \{ |\partial_t F^{\varepsilon, n}_t(\bar x)| \; | \;\bar x \in D_ Q, t \in [0,T] \}\leq C_{Q}.  \end{equation} 
\end{lemma}
\begin{proof} Since $\nabla F^{\varepsilon, n}_t= \nabla^2 W \ast \mu^{\e, n}_t$, the first inequality of the lemma follows as $\nabla^2 W$ is bounded and $\mu^{\e, n}_t$ is a finite measure. It remains to show the second statement which requires a refinement of the proof of Lemma \ref{lem_tightness}. 

Choose $L>0$ such that $|q^\star_L|>Q+a$ and let $\mathcal S,$ $\mathcal G$ and $\mathcal O$ be the sets defined in subsection \ref{sec_defn-set} and corresponding to $q^\star_L$.  Each one of the above sets depends on $\e, n$ as $(p_t, q_t)$ is the flow associated to $v_t^{\e, n}.$ But to alleviate the notation, we don't display these dependences.  Let us provide preliminary estimates for the cases $(p, q) \in \mathcal S, \mathcal G, \mathcal O$: 

If $(p,q) \in \mathcal G$ then $|q_t| > |q^\star_L| >Q+a$ and so, 
\begin{equation}\label{eq:1}
|\nabla W(\bar q-q_t) | |v_t(p_t,q_t)|=0.
\end{equation}

If $(p,q) \in \mathcal S$ then 
\[|\dot p_s| \leq B+ \max_{|q'| \leq |q^\ast_{L}|} |\nabla \Phi(q')|={\mathcal R}(q^\star_{L-1})\] 
for all $s \in [0,t]$ and so, 
\[|v_t(p_t,q_t)|^2= |p_t|^2+ |\dot p_t|^2 \leq \bigl(|p|+T{\mathcal R}(q^\star_{L-1})\bigr)^2 +{\mathcal R}(q^\star_{L-1})^2. \] 
Hence, 
\begin{equation} \label{eq:2} 
|\nabla W(\bar q-q_t) | |v_t(p_t,q_t)| \leq ||\nabla W|| \sqrt{\bigl(|p|+T{\mathcal R}(q^\star_{L-1})\bigr)^2 +{\mathcal R}(q^\star_{L-1})^2}.
\end{equation} 

Finally, consider $(p,q) \in \mathcal O$ and let $L_1$ be such that $|q_{L_1}^\ast | \leq |q| <  |q_{L_1+1}^\ast|$. Note that $L_1 \geq L.$ Suppose that there exists $s \in [0,t)$ such that $|q_s| >  |q_{L_1+1}^\star |$. Let $t_1$ be the smallest such $s$. We have $|q_s| >  |q_{L_1+1}^\star|$ for all $s \in (t_1, t)$ and so, $\nabla W(\bar q-q_t) =0$ for $s \in (t_1, t)$.  If $|q_s| \leq  |q_{L_1+1}^\star|$ for all $s \in [0,t]$ then $|\dot p_s| \leq \mathcal R(q)$ for all $s \in [0,t]$ and so, 
\[|p_t| \leq |p|+t \mathcal R(q). \] 
We therefore have
\[|v_t(p_t,q_t)|^2= |p_t|^2+ |\dot p_t|^2 \leq (|p|+T \mathcal R(q))^2 +\mathcal R(q)^2 \] 
and conclude that 
\begin{equation} \label{eq:3} 
|\nabla W(\bar q-q_t) | |v_t(p_t,q_t)| \leq ||\nabla W|| \sqrt{(|p|+T \mathcal R(q))^2 +\mathcal R(q)^2}.
\end{equation} 

Since by  Equation (\ref{eq:1}) we need not consider $(p, q) \in \mathcal G$ we have
\[A_{t,1}=\int_{\mathcal S \cup \mathcal O} \nabla W(\bar q-q_t) |v_t(p_t,q_t)| R_t(p_t, q_t) ~d\mu_0(p,q).  \]
We can now use Equations (\ref{eq:0}), (\ref{eq:2}) and (\ref{eq:3}) to conclude that 
\begin{equation} \label{eq:first-time-est1} 
\sup_{t , |\bar x| \leq Q}|A_{t,1}(\bar x)|=: C_{1,Q}<\infty.
\end{equation} 
Similar arguments yield 
\begin{equation} \label{eq:first-time-est2} 
\sup_{t , |\bar x| \leq Q}|A_{t,2}(\bar x)|=: C_{2, Q}<\infty.
\end{equation} 
Finally it is noted that these bounds are independent of $\e, n$ since all estimates have been performed with the ``reference'' measure $\mu_0$, regardless of $\e, n$; in particular, while the sets $\mathcal G, \mathcal S, \mathcal O$ themselves may depend on $\e, n$, once the position/momentum bounds have been obtained -- independently of $\e, n$, the measures of these sets are all estimated by the full measure.  This concludes the proof of the theorem.
\end{proof}
 
The required $\e$ tending to zero convergence of velocity field result now follows:
 
 \begin{cor}\label{lim_uc}
Let $K \subset \mathbb R^D$ be a compact set and suppose that $\mu_t^{\e, n} \rightharpoonup \mu_t$ subsequentially for some $\mu_t$ (by this we mean that along some sequence $(\e_k, n_k) \rightarrow (0, \infty)$ we have that $\mu_t^{\e_k, n_k} \rightharpoonup \mu_t$).  Then $v^{\e, n}_t$ converges uniformly to $v_t: = [- \nabla (\Phi + F_t)(q), p]$ on $K \times [0, T]$ (where $F_t = \nabla W * \mu_t$) and consequently, 
\[ v_t^{\e, n} \mu_t^{\e, n} \rightharpoonup v_t \mu_t\]
in the sense of distribution.
\end{cor}
\begin{proof} 
Let first recall that $t_n= t_n(t)$ is the nearest time discretization point to $t$.  It is sufficient to show that $F_{t_n}^{\e, n}$ converges to $F_t$ uniformly (from this the distributional convergence immediately follows).  Let us first observe that $F_{t_n}$ is piecewise constant and (only) agrees with $F_t$ at time discretization points.  

Notwithstanding, we begin by showing that $F_t^{\e, n}$ converges to $F_t$ uniformly.  To this end, we have by Lemma \ref{le:bound-on-times1} that the (phase) spatial and time derivatives are bounded and therefore the family $(F_t^{\e, n})$ is pre--compact on $C(K \times [0, T])$ by the Arzela--Ascoli theorem (and all functions in question are uniformly bounded) so there exists a subsequential limit $\tilde F_t$.  On the other hand, inputing the tightness estimate from Lemma \ref{lem_tightness} (in particular see Remark \ref{e_tight}, (i)) into Lemma \ref{lemma_ham}, we conclude that $\tilde F_t = F_t$, so in particular, we have convergence along the original $(\e_k, n_k)$ sequence.  

Finally, let us take into account the discretization: Since by Lemma \ref{le:bound-on-times1} \[ |F_{t_n}^{\e, n} - F_t| \leq |F_{t_n}^{\e, n} - F_t^{\e, n}| + |F_t^{\e, n} - F_t| \leq C_Q |t_n - t| + |F_t^{\e, n} - F_t|\]
(with $Q > 0$ such that $K \subset D_Q$) and $|t_n - t| \rightarrow 0$,  as $n \rightarrow \infty$, the result follows.
\end{proof}

Now we can repeat the logic of the proof of Theorem \ref{exist_def_ODE} to obtain existence of solution to the (limiting) Hamiltonian ODE: 

\begin{thm}[Existence of Solution to Hamiltonian ODE]\label{exist_ham_ODE}
For fixed $0 < T < \infty$ and $\alpha_0 > 0$ let $\mu_0 \in M_{\alpha_0}$ denote some initial Borel measure on $\mathbb R^D$ of finite total mass and let $\mathscr H$ denote the Hamiltonian in Definition \ref{defn_ham}.  Assume for simplicity that the total mass of $\mu_0$ is $1.$ Then there exists a distributional limit of $(\mu_t^{\e, n})_{t \in [0, T]}$ along some subsequence $(\e_k, n_k)$, denoted $( \mu_t)_{t \in [0, T]}$ starting at $\mu_0$ such that
\begin{itemize}
\item[(i)] $t \rightarrow \mu_t \in \mathscr M$ is distributionally continuous and $M_0(\mu_t) \leq 1$ for $t \in [0,T]$.

\item[(ii)] $(\mu_t)_{t \in [0, T]}$ satisfies the  Hamiltonian ODE 
\begin{equation} \label{def_eqnew} 
\partial_t \mu_t + \nabla \cdot (J \nabla_{W} \mathscr H(\mu_t)\mu_t) = 0.
\end{equation}
\end{itemize}
\end{thm}
\begin{proof}For $\e>0$ we let  $(\mu^{\e, n}_t)_{t \in [0, T]}$ be the time discretized solutions from the proof of Theorem \ref{exist_def_ODE} with initial data $\mu_0$.  It follows from the reasoning in the proof of Proposition \ref{narrow_conv} that we have a distributional limiting curve  $(\mu_t)_{t \in [0, T]} \subset \mathscr M$ which is distributionally continuous; here we have taken $\mu_{t_k}^{\e_\ell, n_\ell} \rightharpoonup \mu_{t_k}$ for a dense set of times in $[0, T]$ and the rest of the argument is identical. 

That $M_0(\mu_t) \leq 1$ follows, after a small argument, from distributional convergence (since $M_0(\mu_t^{\e, n})\leq 1$ for all $\e, n$, by the same reasoning as used in the proof of Theorem \ref{exist_def_ODE}).  Finally, Corollary \ref{lim_uc} gives the necessary convergence of the relevant velocity fields to yield the limiting dynamics (again we refer the reader to the proof of Theorem \ref{exist_def_ODE}).  

\end{proof} 


\subsection{Closeness of Trajectories and Representation Formulae}

In the ensuing subsection we will need stronger properties of $\mu_t^\e$ and associated trajectories.  The key result in this subsection is a pullback formula for the measures $\mu_t^\e$ (Lemma \ref{rep_form}).  As a consequence we will immediately be able to extract a limiting Hamiltonian ODE statement which does not explicitly involve time discretization (Theorem \ref{ett}).  Let us first review the setting of Proposition \ref {prop:erz}:

\bigskip

Let $v^n_t:\mathbb R^D \times [0,T] \rightarrow \mathbb R^D$ be a sequence of Borel maps such that for each compact set $K \subset \mathbb R^D$, the $f_K^n $ are in $L^\infty(0,T)$, where 
\[f_K^n(t)= \sup_{x \in K} |v^n_t| + \text {Lip}(v_t^n,K).\] 
Also let $v_t: \mathbb R^D \times [0, T] \rightarrow \mathbb R^D$ be another Borel map with
\[ f_K(t)= \sup_{x \in K} |v_t| + \text{ Lip}(v_t,K).\]
As in Proposition \ref{prop:erz} we let e.g., $[0, \tau^n(x))$ be the maximal interval on which the ODE
\[ \dot X_t^n = v_t^n(X_t^n), ~~~X_0^n = x \]
admits a unique solution.  

\bigskip

We begin with an abstract closeness of trajectories result.  

\begin{lemma}\label{trajec_lemmanew} 
Let $(v^n_s: \mathbb R^D \times [0, t]  \rightarrow \mathbb R^D \mid n \in \mathbb N)$ be a sequence of velocity fields as described.  Let $t > 0$ and suppose that $v^n_s$ converges uniformly to some limiting velocity field $v_s$ on  $K \times [0, t]$ for any compact set $K \subset \mathbb R^D$.  Suppose further that 
\[
\sup_{n \in \mathbb N} \|f^n_K(t)\|_{L^\infty(0 ,t)} := f_K < \infty.
\]
Let $L > 0$ and define
\[\mathbb B_t(L) = \{ x \in \mathbb R^D \mid X(\cdot, x)[0, t] \subset D_L\}, \]
where $D_L \subset \mathbb R^D$ denotes the (phase space) ball of radius $L$ centered around the origin.
Then given any $\delta > 0$, there exists $n_0(t, L,\delta)$ such that if $n \geq n_0$, 
\[ \sup_{x \in \mathbb B_t(L)}\sup_{s \in [0,t]} |X_s^n(x) - X_s(x)| \leq \delta.\] 
\end{lemma}

\begin{proof} 
Assume on the contrary that the conclusions do not hold. If so, there exists $\delta>0$ and an increasing sequence $(n_k) \subset \mathbb N$, sequences $(x_k) \subset \mathbb B_t(L)$ and $(s_k) \subset [0,t]$ such that 
\begin{equation}\label{ginnew}  |X^{n_k}_{s_k}(x_k)-X_{s_k}(x_k)|>4\delta. 
\end{equation} 
For each $k$ set 
\[ \vartheta^k = \sup_{\tau \geq 0} \left\{ \tau \; | \; |X_s^{n_k}(x_k) - X_s(x_k)| \leq 2\delta~~~\mbox{for all }s \in [0,\tau] \right\}.\]  
Next choose $\sigma>0$ such that 
\begin{equation}\label{choicenew}  
  \sigma \exp\Bigl(f_{D_{L+\delta}}t \Bigr)< f_{D_{L+\delta}} \delta.
\end{equation}
As $(v^n_s)$ converges uniformly to $v_s$ on $D_L \times [0,t]$ we may choose a positive integer $k_0$ so that 
$$
 \sup_{D_L \times [0,t]}   |v^{n_k}_s-v_s| < \sigma
$$ 
for all $k \geq k_0.$

We claim that $\vartheta^k \geq t$ for $k \geq k_0.$ Assume towards a contradiction that $\vartheta^k < t.$ For almost every $s \in (0,\vartheta^k)$ we have 
\[
\begin{split}
\frac{d}{ds} |X_s^{n_k}(x_k) - X_s(x_k)| &\leq |v_s^{n_k}(X_s^{n_k}(x_k)) - v_s(X_s(x_k))|\\
&\leq |v_s^{n_k}(X_s^{n_k}(x_k)) - v_s^{n_k}(X_s(x_k))| + |v_s^{n_k}(X_s(x_k)) - v_s(X_s(x_k))|\\
&\leq f_{D_{L + \delta}}|X_s^{n_k}(x_k) - X_s(x_k)| + \sigma.
\end{split} 
\]
A Gronwall type integral inequality and Equation (\ref{choicenew}) yield   
\[ 
|X_s^{n_k}(x_k) - X_s(x_k)| \leq \frac{\sigma}{f_{D_{L + \delta}}}\cdot\exp(f_{D_{L+\delta}} s)<\delta. 
\]
This proves that $X^{n_k}(x_k, \cdot)[0, \vartheta^k)$ is contained in $D_{L+\delta}.$ Hence the solution $X^{n_k}(x_k, \cdot)$ of the ODE  can be extended to an interval of positive length $[\vartheta^k, \vartheta^k+a_k]$ such that by continuity 
\[
|X_s^{n_k}(x_k) - X_s(x_k)| \leq 2\delta 
\]
on $[0, \vartheta^k+a_k]$. This contradicts the maximality of $\vartheta^k$ and proves the claim. We eventually use the fact that $\vartheta^k \geq t$ for $k \geq k_0$ to contradict Equation (\ref{ginnew}).
\end{proof}

The reasoning behind the above set of ideas now allows us to deduce a representation formula for $\mu_t^\e$ from the representation formula for $\mu_t^{\e, n}$, i.e., the measure at time $t$ is expressible as the push--forward of $\mu_0$, augmented with the depreciation in mass provided by $R_t^\e$.  Let us first recall that 
\[ \mathbb S_t^\e = \{ x \in \mathbb R^D: \tau^\e(x) > t\}, \]
where $\tau^\e(x)$ is such that the ODE $\dot X_t^\e(x) = v_t^\e(X_t^\e), X_0^\e = \mbox{Id}$ has a unique solution on $[0, \tau(x))$ (see Proposition \ref{prop:erz}).  

\begin{lemma}\label{rep_form}
Let $\e > 0, t > 0$ be fixed and let $(\mu_s^\e)_{s \in [0, t]}$ be as constructed in Theorem \ref{exist_def_ODE}.  Then we have the representation formula: If $\varphi \in C_c^\infty(\mathbb R^D)$,
\[ 
\int_{\mathbb R^D} \varphi~d\mu_t^\e = \int_{\mathbb S_t^\e} (\varphi \cdot R_t^\e) \circ X_t^\e~d\mu_0,
\]
where we recall that $\mathbb S_t^\e$ is an open set, by Proposition \ref{prop:erz}, (ii).  
\end{lemma}

\begin{proof}
The result essentially follows from Lemma \ref{trajec_lemmanew} and the fact that since $\e > 0$, there is no ``discontinuity at infinity''.  We first note that by distributional convergence and by the fact that for finite $n$ we do have the representation formula, we have 
\[ \int_{\mathbb R^D} (\varphi \cdot R_t^{\e, n}) \circ X_t^{\e, n}~d\mu_0 = \int_{\mathbb R^D} \varphi~d\mu_t^{\e, n} \longrightarrow \int_{\mathbb R^D} \varphi~d\mu_t^\e, \]
and therefore it is sufficient to establish that 
\begin{equation}\label{to_show} \int_{\mathbb R^D} (\varphi \cdot R_t^{\e, n}) \circ X_t^{\e, n}~d\mu_0 \longrightarrow \int_{\mathbb S_t^{\e}} (\varphi \cdot R_t^{\e}) \circ X_t^{\e}~d\mu_0.\end{equation} 

First we claim that $v^{\e, n}_s$ uniformly converges to $v^\e_s$ on $K \times [0, t]$ for any compact set $K \subset \mathbb R^D$.  Indeed,  it is again sufficient to address the interaction term and the argument is essentially the same as the proof of Corollary \ref{lim_uc}; we remind the reader that the necessary estimates (namely Equations \eqref{vb1} and \eqref{vb2}) to ensure pre--compactness from Lemma \ref{le:bound-on-times1} are uniform in $n$.  We therefore may assume the conclusion of Lemma \ref{trajec_lemmanew} for the trajectories $X_t^{\e, n}$ and $X_t^\e$.

To establish \eqref{to_show} we will divide into two cases:  

\textbf{1.} Case: $x \in \mathbb S_t^\e$ (or $t < \tau^\e(x)$). In this case, for $L$ sufficiently large, $x \in \mathbb B_t^\e(L)$ and therefore by Lemma \ref{trajec_lemmanew} we have the pointwise limit
\[ \lim_{n \rightarrow \infty} (R_t^{\e, n} \cdot \varphi) \circ X_t^{\e, n}(x) = \chi_{_{\mathbb S_t^\e}} (R_t \cdot \varphi) \circ X_t^\e(x).\] 

\textbf{2.} Case: $ x \notin \mathbb S_t^\e$ (or $t \geq \tau^\e(x)$).   Here we claim that pointwise the corresponding portion of the integrand in the left hand side of \eqref{to_show} converges to 0: Choose $r>0$ arbitrary. Then choose $A$ large enough so that $\exp(-\e A)<r.$ Next, choose $\bar t<\tau^\e(x)$ such that $|X_{\bar t}^\e(x)-x|>2A$.  Since $\bar t < \tau^\e(x)$, by Lemma \ref{trajec_lemmanew} there exists $n_0$ such that 
\[|X_{\bar t}^{\e, n}(x)- X_{\bar t}^\e(x)| <A\]
for all $n \geq n_0.$ Hence, 
\[|X_{\bar t}^{\e, n}(x)-x| >A\] 
and so, 
\[  R_t^{\e, n}(X^{\e, n}_t(x))  \leq    \exp\Bigl(-\e \int_0^{\bar t}|\dot X^{\e, n}_s|~ds  \Bigr) 
\leq    \exp\Bigl(-\e \bigl|X^{\e, n}_{\bar t}(x)-x \bigr|~ds  \Bigr) \leq  \exp(-\e A)<   r
 \]
for these $n$. This proves that 
 \[\limsup_{n \rightarrow \infty}  R_t^{\e, n} \circ X^{\e, n}_t(x) \leq  r.
 \]
Taking $r$ to zero, we may conclude the case $x \notin \mathbb S_t^\e$. 

Integrating and invoking the dominated convergence theorem now yields Equation \eqref{to_show}.

\end{proof}

%
 
We can now extract the limiting measures by taking $\e$ to zero (along a sequence) directly:

\begin{thm}\label{ett}
For fixed $0 < T < \infty$ and $\alpha_0 > 0$ let $\mu_0 \in M_{\alpha_0}$ denote some initial Borel measure on $\mathbb R^D$ of finite total mass and let $\mathscr H$ denote the Hamiltonian in Definition \ref{defn_ham}.  Assume for simplicity that the total mass of $\mu_0$ is $1.$ Let $(\mu_t^\e)_{t \in [0, T]}$ be as constructed in Theorem \ref{exist_def_ODE}.  Then there exists a distributional limit of $(\mu_t^{\e})_{t \in [0, T]}$ along some subsequence $(\e_k)$, denoted $( \mu_t)_{t \in [0, T]}$ starting at $\mu_0$ such that
\begin{itemize}
\item[(i)] $t \rightarrow \mu_t \in \mathscr M$ is distributionally continuous and $M_0(\mu_t) \leq 1$ for $t \in [0,T]$.

\item[(ii)] $(\mu_t)_{t \in [0, T]}$ satisfies the  Hamiltonian ODE 
\begin{equation} \label{def_eqnew} 
\partial_t \mu_t + \nabla \cdot (J \nabla_{W} \mathscr H(\mu_t)\mu_t) = 0.
\end{equation}
\end{itemize}
\end{thm}
\begin{proof}
The representation formula from Lemma \ref{rep_form} allows us to adapt the proof of Corollary \ref{lim_uc} for the measures $(\mu_t^\e)_{t \in [0, T]}$, yielding the requisite convergence of velocity fields.  We remind the reader that once the representation formula has been acquired, the key ingredients for the proof of Corollary \ref{lim_uc} are found in Lemma \ref{le:bound-on-times1}.   The conclusion of this Lemma provided derivative bounds on the velocity fields -- Equations \eqref{eq:first-time-est1} and \eqref{eq:first-time-est2} -- and these bounds are uniform in ($n$ and) $\e$.  With these preparatory results in hand, the proof of Theorem \ref{exist_ham_ODE} can be repeated \emph{mutatis mutantis}. 
\end{proof} 

We conclude this subsection with a result which turns out to be of no explicit use in the present work but may be of some independent interest: As an immediate corollary to the preceding ideas we can also deduce a closeness of trajectories result as $\e$ tends to zero.

\begin{cor}\label{trajec_lemma}
Let $L > 0, t > 0$ and consider $\mu_t^\e$ and trajectories given by the dynamics in Theorem \ref{exist_def_ODE}.  For $\e \geq 0$ we define
\[\mathbb B_t^\e(L) = \{ x \in \mathbb R^{D} \mid X^\e(\cdot, x)[0, t] \subset D_L\}, \]
where $D_L$ denotes the phase space ball of radius $L$ centered around the origin.
Suppose $\mu_t^{\e_k} \rightharpoonup \mu_t$.  Then given any $\alpha > 0$, there exists $\e_0(t, L)$ such that if $\e \in (\e_{k})$ and $\e < \e_0$, then 
\[ \sup_{x \in \mathbb B_t^{0}(L)} \sup_{s \in [0,t]} |X_s^\e(x) - X_s(x)| < \delta.\] 
In particular, 
\[\mathbb B_t^\e(L) \subset \mathbb B_t^0(L+ \alpha). \]
\end{cor}
\begin{proof}
It is sufficient to verify the hypothesis of Lemma \ref{trajec_lemmanew}, which is immediate from Lemma \ref{le:bound-on-times1} (here we reiterate that the relevant estimates in Equations \eqref{vb1} and \eqref{vb2} are uniform in $\e$; c.f.,~proof of Theorem \ref{ett}).  We note that a further subsequence in $\e$ is not required as the limit is uniquely specified by Lemma \ref{lemma_ham} (as was the case in the proof of Corollary \ref{lim_uc}). 
\end{proof}

\subsection{Convergence of Mass}

Here, we are in the setting of Theorem \ref{ett}, and we wish to establish statements concerning convergence of (total) mass.  In particular, for $\e_k \rightarrow 0$ and $\mu_t^{\e_k} \rightharpoonup \mu_t$ we will show that a limit exists for the finite $\e$--masses and in particular it agrees, a.e., with the mass of the limiting measure.

Since $\mu^\e_t$ has been constructed in Theorem \ref{exist_def_ODE} as the limit for the \emph{narrow} convergence of a sequence $(\mu^{\e,n}_t)$ such that $t \rightarrow \mu_t^{\e, n}(\mathbb R^D)$ is monotone nonincreasing, $t \rightarrow \mu_t^{\e}(\mathbb R^D)$ is monotone nonincreasing (which also follows from the fact that $\mu_t^\e$ satisfies Equation \eqref{partial_cont}). Unfortunately, in Theorem \ref{ett}, $\mu_t$ is  obtained as a limit of a subsequence of $ (\mu_t^{\e})$ only for the \emph{vague} topology and so, the above simple argument does not apply.  What is however obvious is that as $\mu_0^{\e}$ is independent of $\e$, $\mu_t(\mathbb R^D)\leq \mu_0(\mathbb R^D)$. 

In light of the previous comments, we plan to first demonstrate that, at least under the dynamical condition, in the limiting measure the mass can only decrease in time. (That is, there cannot be particles returning from infinity.)  

First, for expository ease, we will introduce \emph{compact} regions of no return in phase space.  

\begin{prop}\label{prop_psb}
Let $t > \bar t \geq 0$ and consider a trajectory $X_t = (p_t, q_t)$ satisfying the dynamics as in Theorem \ref{exist_def_ODE} such that $X_{\bar t}$ is in some $B_{L'} \times B_{L^\star}$ for some $L' > 0$ and where  $L^\star$ is as in Hypothesis \ref{dyn_hyp} \emph{and} remain in $\mathbb R^d \times B_{L^\star} = \Omega_{L^\star}$ up to time $t$.  Then there is a $L^\star$--dependent constant $a_\star$ such that uniformly in $\e$, for all $\tau \in [\bar t, t]$, 
\[ |p_\tau| \leq a_\star (t-\bar t) + |p_{\bar t}|.\]
\end{prop}
\begin{proof}
Let us write $\tilde F_t^\e(q_t) = \nabla \Phi(q_t) + \nabla W * \mu_t^\e(q_t)$ and observe that since $|\nabla W|$ and $|\nabla \Phi|$ are both bounded in $\Omega_{L^\star}$, there is some $L^\star$--dependent constant $a_\star$ such that $|F_t(q_t)| \leq a_\star$.  Taking into account the fact that, for all $\e$, the total mass of $ \mu_t^\e$ can, without loss of generality, be assumed to be less than or equal to $1$, we may explicitly  choose 
\[ a_\star =  \sup_{q \in B_{L^\star}}|\nabla \Phi|+  |\nabla W|. \]
We have by the canonical equations that (uniformly in $\e$)
\[\frac{d}{ds}|p_s| \leq |\dot p_s| = |\tilde F_s(q_s)| \leq a_\star,\]
where we have used that $|q_s| \leq |q_\star| = L^\star$ for all $0 \leq s \leq t.$ The result follows by integration.
\end{proof}

It is now clear that we may define phase space regions of no return:

\begin{defn}\label{defn_psb}
Let $T > 0$ and let $(\mu_t^\e)_{t \in [0, T]}$ be given as in Theorem \ref{exist_def_ODE} and let $L^\star \rightarrow \infty$ be the sequence as in Hypothesis \ref{dyn_hyp}.  Let $\eta > 0$ be an arbitrary (small) number.  Then we define
\[ \bar \Omega_{L^\star}(t) = B_{L^\star + (a_\star + \eta) t} \times B_{L^\star}, \]
where $a_\star$ is as in Proposition \ref{prop_psb}.
\end{defn}

We then have 
\begin{lemma}\label{lemma_psb}
Let $\bar t \in  [0, T]$ be fixed.  Under Hypothesis \ref{dyn_hyp}, for any trajectory $(p_{\bar t}, q_{\bar t}) \in \bar \Omega_{L^\star}(\bar t)$ satisfying the dynamics as given in Theorem \ref{exist_def_ODE}, it is the case that either 
\[ (p_t, q_t) \in \bar \Omega_{L^\star}(t), ~~~\mbox{for all $t \geq \bar t$}, \]
or 
\[ |q_t| \rightarrow \infty\]
-- either at some finite time or as $t\to\infty$.  In particular, there is some time
$t_{*}$ after which the radial speed is uniformly bounded away from zero:
i.e., there exists some $\alpha$ and some $t_* > \bar t$ such that 
\[ \frac{d|q_t|}{dt} \geq \alpha > 0, ~~~\mbox{for all } t \geq t_*.\]
Moreover, $(p_t, q_t)$ exits $\bar \Omega_{L^\star}(t)$ on the $|q| = |q^\star|$ boundary.

\end{lemma}
\begin{proof}
This follows immediately from Lemma \ref{rad_outward} and Proposition \ref{prop_psb}. Indeed if the space marginal never exceeds $L^\star$, then by Proposition \ref{prop_psb} the momentum marginal remains within the stated bounds: Explicitly,  as long as $q_{s'} \in B_{L^\star}$ for all $s' \in [\bar t,s]$, the corresponding momentum satisfies
\[ |p_s| \leq L^\star + \bar t(a_\star + \eta) + (s - \bar t) a_\star < L^\star + s(a_\star + \eta)\]
and hence $(p_s, q_s) \in \bar \Omega_{L^\star}(s)$.  It follows that the only available exit is via the \emph{position} space marginal and hence no possibility of return, by Lemma 2.1.
\end{proof}

\begin{prop} \label{mmm}
Let $T > 0$ and let $(\mu_t)_{t \in [0, T]}$ be as obtained in Theorem \ref{ett} and let us denote by $\mathbb M_t$ the total mass at time $t$:  $\mathbb M_t = \mu_t(\mathbb R^D)$.  Then $\mathbb M_t$ is monotone non--increasing in $t$.
\end{prop}
\begin{proof}
Suppose $0 \leq t_1 < t_2 \leq T$.  Let $\delta > 0$ and let us choose $L^\star > 0$ sufficiently large where $L^\star$ is as described before so that 
\[ 
\mu_0((D_{L^\star})^c) < \delta
\]
(where $D_{L^\star}$ denotes the \emph{phase space} ball of radius $L^\star$) and 
\[
\mu_{t_2}\left(     \overline{(  \bar \Omega_{L^\star}(t_2)   )^c}\right) < \delta.
\]
Now we claim that for all $\e$, 
\begin{equation}\label{mmmon}
\mu_{t_1}^\e(\bar \Omega_{L^\star}(t_1)) \geq \mu_{t_2}^\e (\bar \Omega_{L^\star}(t_2)) - \delta.
\end{equation}
In broad strokes, the argument proceeds as follows: By the representation formula in Lemma \ref{rep_form}, we can decompose (at time $t_2$) the mass in $\bar \Omega_{L^\star}(t_2)$ into that which initiated, at $t = 0$, from $\bar \Omega_{L^\star}(0) (= D_{L^\star})$ and that which did not.  The latter clearly has $\mu_0$ mass bounded by $\delta$, while the former, path--wise, must be in the set $\bar \Omega_{L^\star}(t_1)$ at time $t_1$, by the no--return condition stated in Lemma \ref{lemma_psb}.  

More explicitly, we claim that
\begin{equation}\label{set_cont}
D_{L^\star} \cap (X_{t_2}^\e)^{-1}(\bar\Omega_{L^\star}(t_2))  \subset D_{L^\star} \cap (X_{t_1}^\e)^{-1}(\bar \Omega_{L^\star}(t_1)).
\end{equation}
Indeed, suppose $(p_0, q_0) \in D_{L^\star}$ and $(p_{t_2}, q_{t_2}) \in \bar\Omega_{L^\star}(t_2)$. If $|q_t|>L^\star$ occurs for some $t \in (0,t_2)$, the no return condition would yield that $|q_{t_2}|>L^\star$, which contradicts the fact that $(p_{t_2}, q_{t_2}) \in \bar\Omega_{L^\star}(t_2)$. Consequently, $|q_t|\leq L^\star$ for all $t \in [0,t_2]$ and so, $|d p_t/dt| \leq a_\star$ for $t \in [0,t_1]$ which yields $(p_0, q_0) \in (X_{t_1}^\e)^{-1}(\bar \Omega_{L^\star}(t_1)).$ 

The claimed inequality Equation \eqref{mmmon} now basically follows from the above set containment along with the fact that $R_t^\e$ is decreasing in $t$ along each trajectory.  More precisely, let $\varphi_k \in C_0$ with values in $[0, 1]$ and support in $\pc{L^\star}{t_2}$ be functions that satisfy 
\[\mu_{t_2}^\e((\pc{L^\star}{t_2})^\circ) = \lim_{k \rightarrow \infty} \int_{\pc{L^\star}{t_2}} \varphi_k~d\mu_{t_2}^\e.\]  
(Such functions are readily constructed.)   Then 
\[\begin{split} \int_{\pc{L^\star}{t_2}} \varphi_k~d\mu_{t_2}^\e &= \int_{(X_{t_2}^\e)^{-1} (\pc{L^\star}{t_2})} (\varphi_k ~\cdot ~R^\e_{t_2}) ~\circ~ X_{t_2}^\e~d\mu_0\\
&\leq \int_{D_{L^\star} \cap (X_{t_2}^\e)^{-1} (\pc{L^\star}{t_2}) } (\varphi_k~ \cdot ~R_{t_2}^\e) ~\circ~ X_{t_2}^\e~d\mu_0 + \delta\end{split}\]
by our choice of $D_{L^\star}$.  Invoking the set containment in Equation \eqref{set_cont}, the above inequality can be continued as 
\[ \int_{\pc{L^\star}{t_2}} \varphi_k~d\mu_{t_2}^\e\leq \int_{D_{L^\star} \cap (X_{t_1}^\e)^{-1} (\bar \Omega_{L^\star}(t_1))} (\varphi_k ~\cdot~ R_{t_2}^\e)~ \circ ~X_{t_2}^\e~d\mu_0 + \delta.\]
Since $R_{t_2}^\e \circ X_{t_2}^\e \leq R_{t_1}^\e \circ X_{t_1}^\e$ the above becomes 
\[ \int_{\pc{L^\star}{t_2}} \varphi_k~d\mu_{t_2}^\e\leq \int_{D_{L^\star} \cap (X_{t_1}^\e)^{-1} (\bar \Omega_{L^\star}(t_1))} (\varphi_k \circ X_{t_2}^\e)  ~\cdot~ (R_{t_1}^\e\circ X_{t_1}^\e)~d\mu_0 + \delta.\]
Next we observe that pushing forward to time $t_1$, the second term in the product in the integrand together with $d\mu_0$ becomes $d\mu_{t_1}^\e$, the set of integration becomes $\pc{L^\star}{t_1} \cap X_{t_1}^\e(D_{L^\star})$, and the integrand becomes $\varphi_k \circ X_{t_2 - t_1}$:
\[ \int_{\pc{L^\star}{t_2}} \varphi_k~d\mu_{t_2}^\e\leq \int_{\pc{L^\star}{t_1} \cap X_{t_1}^\e(D_{L^\star})} (\varphi_k \circ X_{t_2 - t_1})~d\mu_{t_1}^\e + \delta  \leq \mu_{t_1}^\e(\pc{L^\star}{t_1}) + \delta,\] 
where the last inequality follows from the fact that $\varphi_k \leq 1$.  Taking $k$ to infinity, we conclude 
\[ \mu_{t_2}^\e((\pc{L^\star}{t_2})^\circ) \leq \mu_{t_1}^\e(\pc{L^\star}{t_1}) + \delta. \]

Since the above holds for all $\e$, we have by Equation \eqref{set_conv} and the choice of $L^\star$ that 
\[\begin{split}  \mathbb M_{t_1} &\geq \mu_{t_1} (\bar \Omega_{L^\star}(t_1)) \geq \limsup_\e \mu_{t_1}^\e(\pc{L^\star}{t_1})
\geq \limsup_\e \mu_{t_2}^\e((\pc{L^\star}{t_2})^\circ) - \delta\\
&\geq \liminf_\e \mu_{t_2}^\e((\pc{L^\star}{t_2})^\circ) - \delta \geq \mu_{t_2}((\pc{L^\star}{t_2})^\circ) - \delta \geq \mathbb M_{t_2} - 2\delta \end{split}\]
and the desired monotonicity follows by taking $\delta$ to zero.
\end{proof}


Now we introduce a more quantitative version of the dynamical condition which can be understood as a requirement that the external potential diverges sufficiently fast near infinity:

\begin{hyp}\label{quant_dyn_cond} 
As in Section \ref{d}, we consider dynamics given by the Hamiltonian 
\begin{equation}\label{ham_recall}H(p, q,t) = \frac{1}{2} |p|^2 + \Phi(q)+ \Psi(t,q),\end{equation}
where $|\nabla \Psi(t, q)| < B, \forall t \geq 0, q \in \mathbb R^d$ (so that in particular we may take, e.g., $\Psi (t, q) = (W * \mu_t^\e)(q_t)$).  Recall that Hypothesis \ref{dyn_hyp} concerned the existence of spherically symmetric bounding potentials $\Upsilon$ such that $\Upsilon (q)< \Upsilon(q_L^\star)$ for all $|q|> |q^\star_L|$.  

Here we are concerned with \emph{pairs} of position space rings $L$ and $\ell(L)$ such that $\ell(L) < L$ (not necessarily adjacent). Let $t > 0$ be essentially arbitrary and let us consider trajectories of particles operating under $\Upsilon$--dynamics which 
exit $B_L$ at time $t$, having at some earlier time exited $B_\ell$
$$
\tilde E_{L}(t) :=  \{q_s\mid q_t \in \partial B_L,
\hspace{.1 cm}
 q_{t'} \in B_{\ell(L)} \mbox{~~for some~~} t' < t.
\}
$$
Obviously, for some $t$'s, the sets $\tilde E_{L}(t)$ are non--empty.  For $\tilde E_{L}(t) \neq \emptyset$ we may define
$$
\tilde \vartheta_{L}(t)  =  \sup\{\tau\mid
q(s)\in \tilde E_{L}(t), |q_{t + \tau}| < \infty
\}.
$$
And, if $\tilde E_{L}(t) = \emptyset$ -- e.g., if $\Upsilon$ is very repulsive and $t$ is too large -- then we may, for connivence, define $\tilde \vartheta_{L}(t) = 0$.  Finally we define
$$
\tilde \tau_L   =  \sup_{t}\tilde \vartheta_{L}(t).
$$


\emph{We take as a hypothesis the existence of a sequence $(L, \ell(L))$ with $\ell(L) \rightarrow \infty$ such that}
\[ \lim_{L \rightarrow \infty} \tilde \tau_L = 0.\]
\end{hyp}

\begin{remark}
Following along the lines of the discussions in Remark \ref{DSN}, it is readily derived that if the bounding potential $\Upsilon$ satisfies power law upper and lower bounds of the form 
\[ D_2 |q|^{d_2} \leq |\Upsilon| \leq D_1 |q|^{d_1}\]
with $d_{1}$, $d_{2}$ larger than 2
then, if Hypothesis \ref{dyn_hyp} is satisfied, then the stronger Hypothesis \ref{quant_dyn_cond} also holds. 
\end{remark}

\begin{prop}\label{poof}
Let $t > 0$ be arbitrary and consider Hamiltonian dynamics according to Equation \eqref{ham_recall}.  Suppose Hypothesis \ref{quant_dyn_cond} holds.  Let us define the \emph{phase space} ``escape times'' $\tau_{L}$
by analogy to the above with $B_{L}$'s etc., replaced by the appropriate $\bar{\Omega}_{L}$'s providing us with (untilded) versions of $\tilde E_{L}(t)$ and $\tilde \vartheta_{L}(t)$.

Then we have
\[ \lim \tau_L \rightarrow 0 ~~~\mbox{as~~~} L \rightarrow \infty.\]
\end{prop}

\begin{proof}
This is immediate from Lemma \ref{lemma_psb} which ensures that when particles exit $\bar \Omega_L(t)$, they do so via the position space barrier and the fact (as can be seen from e.g., Equation \eqref{e:second-der-of-|q|}) that the actual radial momentum is bounded by that provided by the $\Upsilon$--dynamics.
\end{proof}

\begin{remark}
We remark that for finite $\varepsilon$, the quantity $\tau_{L}$ is a universal bound for the non--existence of trajectories that are in $\bar\Omega_{\ell}(t')$ at time $t'$ and exiting $\bar\Omega_{L}(t)$ at time $t$.  Indeed, all dynamics 
are bounded by the dynamics driven by the potential $\Upsilon$ which is, by fiat, uniform in $\varepsilon$.
\end{remark}

\begin{thm}\label{mff}
Suppose Hypothesis \ref{quant_dyn_cond} holds and suppose that $\mu_t^{\e_k} \rightharpoonup \mu_t$ as in Theorem \ref{ett}.  Then it is the case that for almost every $t \in [0, T]$,
\[\mathbb  M_t = \lim_{\e_k \rightarrow 0} \mathbb M_t^{\e_k}.\]  
More specifically the above convergence holds at all points of continuity of $\mathbb M_t$.
\end{thm}

\begin{proof}
Let us denote 
\[ \mathbb M_{t}^- = \lim_{t' \nearrow t} \mathbb M_{t'},~~~\mathbb M_{t}^+ = \lim_{t' \searrow t} \mathbb M_{t'}\]
and 
\[ \mathbb M_t^\bullet = \limsup_{\e_k \rightarrow 0}~ \mathbb M_t^{\e_k},~~~\mathbb M_t^\circ = \liminf_{\e_k \rightarrow 0}~\mathbb M_t^{\e_k}. \]
It is clear that $\mathbb M_t^\circ \leq \mathbb M_t^\bullet$ and the monotonicity result of $\mathbb M_t$ established in Proposition \ref{mmm} gives that $\mathbb M_t^+ \leq \mathbb M_t^-$.  We will establish that in fact
\[ \mathbb M_t^+ \leq \mathbb M_t^\circ \leq \mathbb M_t^\bullet \leq \mathbb M_t^-\]
from which the result follows.  In particular (although (sub)subsequential limits are already guaranteed by the monotonicity of the $\mathbb M_{t}^{\varepsilon_k}$) this establishes, a.e., the existence of the limit for $\mathbb M_t^{\e_k}$ at points of continuity of $\mathbb M_t$.  We will separate the proof into two claims.

\noindent\textbf{Claim.}
$\mathbb M_t^+ \leq  \mathbb M_t^\circ$:

Let $\delta > 0$ be arbitrary, and let $L > 0$ be such that $\mu_t(\overline{D_L^c}) < \delta$, where $D_L$ denotes a \emph{phase space} ball of radius $L$.  Then, from weak$^*$ convergence (see Equation \eqref{set_conv})
\[ \mathbb M_t \leq \mu_t(D_L^\circ) + \delta \leq \liminf_{\e_k \rightarrow 0} \mu_t^{\e_k}(D_L) + \delta \leq M_t^\circ + \delta\]
and the claim follows.  
(Note that this shows that a mass convergence result is immediate in the absence of the interaction $W$, since then all trajectories $X_t^{\e}$'s are the same and the masses $\mathbb M_t^{\e}$ are monotonically \emph{increasing} as $\e \rightarrow 0$.)

\noindent\textbf{Claim.}
$\mathbb M_t^\bullet \leq \mathbb M_t^-$:

Let $\delta > 0$ and let $\ell := \ell(L)$ be from the hypothesized sequence in Hypothesis \ref{quant_dyn_cond} such that 
\[ \mu_0\left(\overline{(\pc{\ell}{0})^c}\right) < \delta. \]
Next given any $\e > 0$ we let $L_\e > 0$ be such that 
\[ \mu_t^\e\left(\overline{(\pc{L_\e}{t})^c}\right) < \delta.\]
Let us now consider the time $t - \tau_L$ with $\tau_L$ as in Proposition \ref{poof}.  Then we claim that 
\[ \mu_t^\e(\pc{L_\e}{t}) \leq \mu_{t - \tau_L}^\e(\pc{L}{t - \tau_L}) + \delta.\]
On the level of heuristics (and neglecting for the time being the (beneficial) effect of reduction in mass afforded by $R_t^\e$) the above display can be understood as follows: First we observe that by our choice of $\ell$ we may restrict attention to mass that \emph{initiated} in $D_\ell$.  Now all of this mass that is in $\pc{L}{t}$ at time $t$ is certainly in $\pc{L}{t'}$ for $t' < t$ since $L$ is a ring of no return (this is the same reasoning as used in the proof of Proposition \ref{mmm}) and in particular this applies to $t' = t - \tau_L$.  As for the mass in $\pc{L_\e}{t} \setminus \pc{L}{t}$, we note that if the representative particles had already left $\pc{L}{t - \tau_L}$ before time $t - \tau_L$ then by time $t$ they would be (well) beyond $\pc{L_\e}{t}$, by the definition of $\tau_L$; here we are specifically employing the hypothesized properties of $(\ell, L)$.  

The actual proof proceeds as follows: Let $\eta > 0$ and let $\varphi_\eta$ be a continuous function such 
$\varphi_\eta = 1$ on $\pc{L_\e}{t}$ and $\varphi_\eta  = 0$ on $\overline{\left(\pc{L_\e + \eta}{t}\right)^c}$.  We have then by the representation formula in Lemma \ref{rep_form}
\[ \begin{split}
\mu_t^\e(\pc{L_\e}{t}) & \leq \int_{\mathbb R^D} \varphi_\eta~d\mu_t^\e = \int_{\mathbb R^D} (\varphi_\eta ~\cdot~ R_t^\e)~\circ~X_t^\e~d\mu_0\\
&\leq \int_{(X_t^\e)^{-1}(\pc{L_\e + \eta}{t}) \cap D_\ell}~~ (\varphi_\eta ~\cdot~ R_t^\e)~\circ~X_t^\e~d\mu_0 ~~+~~ \delta.
\end{split}\]

Now, we claim, we have the set containment 
\[ D_\ell \cap (X_t^\e)^{-1} (\pc{L_\e + \eta}{t}) \subset D_\ell \cap (X_{t - \tau_L}^\e)^{-1} (\pc{L}{t - \tau_L}).\]
Indeed, following the reasoning in Equation \eqref{set_cont} we certainly have that the left hand side is contained in $D_\ell \cap (X_{t - \tau_L}^\e)^{-1} (\pc{L_\e + \eta}{t - \tau_L})$, so it remains to establish the stronger statement that we can shrink down to spatial scale $L$.  Suppose then that $(p_0, q_0) \in D_\ell$ and 
\[(p_{t - \tau_L}, q_{t - \tau_L}) \in \pc{L_\e + \eta}{t - \tau_L} \setminus \pc{L}{t - \tau_L}. \] 
Then the trajectory was, initially, in $\bar\Omega_{\ell}(0)$ and at some time $s'$ which is earlier than
$t - \tau_{L}$ had exited $\bar\Omega_{L}(s')$.  It follows by the definition of $\tau_{L}$ that at some point  before 
time $t$, the trajectory had ceased to exist (gone to infinity) and therefore it is certainly not in $\pc{L_\e + \eta}{t}$.

Continuing and using the representation formula from Lemma \ref{rep_form} again, we now have 
\[\begin{split} \mu_t^\e(\pc{L_\e}{t}) &\leq \int_{(X_{t - \tau_L}^\e)^{-1} (\pc{L}{t - \tau_L}) \cap D_\ell} ~~ (\varphi_\eta ~\cdot~R_t^\e)~\circ~X_t^\e~d\mu_0 + \delta\\
&\leq \int_{\pc{L}{t - \tau_L}}~~ \varphi_\eta ~d\mu_{t - \tau_L}^\e + \delta\\
&\leq \mu_{t - \tau_L}^\e(\pc{L}{t - \tau_L}) + \delta
\end{split}\]
as claimed.

Using the inequality established above and the choice of $L_{\e_k}$ we have 
\[ \mathbb M_t^{\e_k} \leq \mu_t^{\e_k}(\pc{L_{\e_k}}{t}) + \delta \leq \mu_{t - \tau_L}^{\e_k}(\pc{L}{t - \tau_L}) + 2\delta.\]
Taking the $\limsup$ and recalling again Equation \eqref{set_conv}, we  arrive at 
\[ \mathbb M_t^\bullet \leq \mu_{t - \tau_L} (\pc{L}{t - \tau_L}) + 2\delta \leq \mathbb M_{t - \tau_L} + 2\delta.\]
The result follows by taking $L$ to infinity.

\end{proof}

\section*{Acknowledgements}

The authors would like to thank IPAM for its hospitality during the Optimal Transport Workshop (Spring `08) wherein this project was conceived.  L.~C.~was supported in part by the NSF under the grant DMS--08--05486.  W.~G.~was supported in part by the NSF under the grants DMS--06--00791 and DMS--0901070.  H.~K.~L.~was supported in part by the Graduate Research Fellowship Program and the Dissertation Year Fellowship Program at UCLA and by the NSF under the grants DMS--10--04735 and DMS--08--05486.

\end{document}